\documentclass[11pt]{amsart}
\usepackage{amssymb}
\usepackage{enumerate}
\usepackage{xcolor}
\usepackage{microtype}
\usepackage[plainpages=false,hypertexnames=false,pdfpagelabels]{hyperref}
\usepackage{multirow}
\usepackage{graphicx}
\usepackage{pdflscape}
\usepackage{breqn}
\usepackage[T1]{fontenc}
\usepackage{microtype}
\usepackage{libertine}

\newcommand{\arxiv}[1]{\href{http://arxiv.org/abs/#1}{\tt arXiv:\nolinkurl{#1}}}
\newcommand{\doi}[1]{\href{http://dx.doi.org/#1}{{\tt DOI:#1}}}

\theoremstyle{plain}
\newtheorem{prop}{Proposition}[section]

\newtheorem{thm}[prop]{Theorem}

\newtheorem{lem}[prop]{Lemma}
\newtheorem{cor}[prop]{Corollary}
\newtheorem{Def}[prop]{Definition}

\newenvironment{psmallmatrix}
  {\left(\begin{smallmatrix}}
  {\end{smallmatrix}\right)}

\def\semicolon{;}
\def\applytolist#1{
    \expandafter\def\csname multi#1\endcsname##1{
        \def\multiack{##1}\ifx\multiack\semicolon
            \def\next{\relax}
        \else
            \csname #1\endcsname{##1}
            \def\next{\csname multi#1\endcsname}
        \fi
        \next}
    \csname multi#1\endcsname}

\makeatletter
\newcommand*{\transpose}{%
  {\mathpalette\@transpose{}}%
}
\newcommand*{\@transpose}[2]{%
  \raisebox{\depth}{$\m@th#1\intercal$}%
}
\makeatother

\def\calc#1{\expandafter\def\csname c#1\endcsname{{\mathcal #1}}}
\applytolist{calc}QWERTYUIOPLKJHGFDSAZXCVBNM;
\def\bbc#1{\expandafter\def\csname bb#1\endcsname{{\mathbb #1}}}
\applytolist{bbc}QWERTYUIOPLKJHGFDSAZXCVBNM;
\def\bfc#1{\expandafter\def\csname bf#1\endcsname{{\mathbf #1}}}
\applytolist{bfc}QWERTYUIOPLKJHGFDSAZXCVBNM;

\DeclareMathOperator{\Gal}{Gal}

\newcommand{\iso}{\cong}

\usepackage{fullpage} 

\title{Modular data for the extended Haagerup subfactor}
\author{Terry Gannon and Scott Morrison}   
\date{}

\begin{document}

\begin{abstract}
We compute the modular data (that is, the $S$ and $T$ matrices) for the centre of the extended Haagerup subfactor
\cite{MR2979509}. The
full structure (i.e. the associativity data, also known as 6-$j$ symbols or $F$ matrices) still appears to be 
inaccessible. Nevertheless, starting with just the number of simple objects and their dimensions (obtained by a
combinatorial argument in \cite{1404.3955}) we find that it is surprisingly easy to leverage knowledge of the
representation theory of $SL (2, \bbZ)$ into a complete description of the modular data. We also investigate the possible
character vectors associated with this modular data.
\end{abstract}

\maketitle
  
\section{Introduction}
The extended Haagerup subfactor provides perhaps the strangest currently known example of a \emph{quantum symmetry}. 

Fusion categories provide a suitable axiomatization for the notion of quantum symmetry: they are the finitely semisimple rigid tensor categories. The fundamental examples are the representation categories of finite groups (over $\mathbb C$), but there are many others. The semisimplified
representation category of a quantum enveloping algebra $U_q \mathfrak{g}$ at a suitable root of unity gives another source of examples.

The remarkable discovery of an interesting classification of finite depth subfactors above index 4, initiated by Haagerup \cite{MR1317352}, began to provide examples beyond these `classical' ones. In particular, each finite depth subfactor $N \subset M$ gives a pair of Morita equivalent unitary fusion categories, as the categories of $N-N$ and $M-M$ bimodules. Haagerup and Asaeda constructed `exotic' subfactors in \cite{MR1686551}, and the
last missing case in Haagerup's classification between index 4 and $3+\sqrt{3}$ was provided by the construction by Bigelow-Morrison-Peters-Snyder of the extended Haagerup subfactor  \cite{MR2979509}. Some of these fusion categories are distinctly different from those arising from finite groups or quantum groups: in particular the fusion categories coming from the Haagerup and extended Haagerup subfactors cannot be defined over any cyclotomic field \cite{MR2922607}.

Since the discovery of these examples, there has been some progress towards organising them. In particular, the theory 
of \emph{quadratic} categories has been developed, particularly by Izumi \cite{MR1832764,1512.04288} and Evans-Gannon 
\cite{MR2837122,MR3167494}. These are categories with a group of invertible objects, and under the action of this group 
by left and right tensor product, just one other double coset. The category of $N-N$ bimodules of the Haagerup subfactor
is a quadratic category. While the fusion categories coming from the Asaeda-Haagerup subfactor are not quadratic, work 
of Grossman-Izumi-Snyder \cite{1501.07324} shows that they are Morita equivalent to quadratic categories.

This leaves us with the following remarkable observation: the extended Haagerup fusion categories are the only known 
fusion categories not known to be related to finite groups, quantum groups, and quadratic categories. While this almost 
surely only reflects our feeble ability to discover and construct fusion categories, nevertheless these
categories remain uniquely interesting objects.

Every fusion category has a braided centre, which is a modular tensor category. This paper tackles the problem of 
describing the braided centre of the extended Haagerup categories. While we do not give a full description (in 
particular the associators), we produce the modular data, that is, the $S$ and $T$ matrices.

Recently, Morrison-Walker discovered  \cite{1404.3955} that a purely combinatorial argument determines the number of simple 
objects, and their dimensions, in the centre of extended Haagerup. This paper uses that just that information, and by 
representation theoretic arguments determines the modular data. 

More generally, fusion categories are notoriously difficult to classify, and we hope that the methods described here can
be developed into part of a machine for analysing potential new examples. As a precedent, the classifications of rank 2
and of rank 3 fusion categories \cite{MR1981895,1309.4822} have relied heavily on understanding the possible modular
centres. In fact, the arguments Sections \ref{sec:galois-action}, \ref{sec:group-of-12}, and \ref{sec:group-of-4} have been automated as part of a developing {\tt Mathematica} package, which for example can also perform the analogous arguments for the Haagerup and Asaede-Haagerup categories.

It seems likely that every unitary modular tensor category can be realised as the representation category for some 
strongly rational vertex operator algebra, and as such a CFT would offer at least some `explanation' for the existence 
of the extended Haagerup subfactor. We explore what can be said about such an object. In particular, we are able to 
describe the possible character vectors associated to such a CFT. For $c=8$ or $c =16$, we can completely 
enumerate them; for $c =24$ we at least show that there are plausible candidates.

Both the Haagerup and extended Haagerup subfactors see the prime 13. Is this a coincidence?
The 13 enters their modular data in apparently different ways: through the inequivalent irreps we call $\rho^{(13)}_5$ 
and $\rho^{(13)}_{14}$ for the
Haagerup and the extended Haagerup respectively. However $\rho^{(13)}_{14}$ lies in the symmetric square
of  $\rho^{(13)}_5$, and we will see in Section \ref{sec:EH-characters} that the possible character vectors for both 
(at the smallest possible value of central charge, namely $c=8$) are built from theta functions of the lattice  
$L=A_{3}{52}[1,\frac{1}{4}]$, using notation of \cite{MR1662447}.
In particular, our work suggests that there may be a natural relation between the (still hypothetical) 
extended Haagerup VOA $\cV_{EH}$ and the square $\cV_{Haag}\otimes
\cV_{Haag}$ of the (still hypothetical)  Haagerup VOA.

\section{Background}

Throughout we write $\xi_m=e^{2\pi i/m}$, $\bbZ_N=\bbZ/N\bbZ$.  

\subsection{From subfactors to modular tensor categories}

A \textit{fusion category} $\mathcal{C}$ is a $\bbC$-linear semi-simple rigid monoidal category with finitely many isomorphism classes of simple objects and finite-dimensional spaces of morphisms, such that the endomorphism algebra of the unit object 1 is $\bbC$. A $*$-operation on $\mathcal{C}$ 
is a conjugate-linear involution Hom$(x,y)\rightarrow \mathrm{Hom}(y,x)$ satisfying $(fg)^*=g^*f^*$ and
 $(f\otimes h)^*=f^*\otimes h^*$  for all $f\in\mathrm{Hom}(x,y)$, $g\in\mathrm{Hom}(z,y)$ and $h\in\mathrm{Hom}(z,w)$.
A $*$-operation is called positive if $f^*f=0$ implies $f=0$. A category equipped with a positive $*$-operation is called \textit{unitary} or $C^*$.

Given a finite index and depth subfactor $N\subset M$ of Type II$_1$ factors, we obtain two unitary fusion categories: the \textit{principal even part} consisting of the $N$-$N$ bimodules which occur as summands of
tensor powers of $_NM_N$, and the \textit{dual even part}, consisting of the $M$-$M$ bimodules occurring as summands of tensor powers of  $_MM\otimes_NM_M$.

Let $\mathcal{C}$ be any fusion category. 
Write $\Phi(\mathcal{C})$ for its set of isomorphism classes of simple objects. So rank$\,\cC=\|\Phi(\cC)\|$.
The Grothendieck ring $K(\mathcal{C})$ of $\mathcal{C}$ is also called its fusion ring. Given 
$[x],[y]\in\Phi(\mathcal{C})$,   the structure constants $N_{[x],[y]}^{[z]}\in\bbZ_{\ge 0}$ of the fusion ring
defined by $[x][y]=\sum_{[z]}N_{[x],[y]}^{[z]}[z]$ are called the \textit{fusion coefficients}. A \textit{dimension}
on $\mathcal{C}$ is a ring homomorphism from the fusion ring $K(\mathcal{C})$ to $\bbC$; the \textit{Perron-Frobenius dimension} PFdim of $\mathcal{C}$ is the unique dimension taking
positive real values on all non-zero objects. 

A \textit{modular tensor category} is a spherical braided fusion category $\mathcal{C}$ satisfying a certain nondegeneracy condition. Define a matrix $\tilde{S}$, with rows and columns indexed by $\Phi(\mathcal{C})$,
 by $\tilde{S}_{[x],[y]}=\mathrm{tr}_{x\otimes y}(c_{y,x}\circ c_{x,y})$, where $c_{x,y}$ is the braiding. Then $\tilde{S}$ is well-defined; the non-degeneracy condition is that $\tilde{S}$
  be invertible.  In a modular tensor category, $\tilde{S}$ is
  symmetric, and the values $\tilde{S}_{[1],[x]}$ define a dimension {dim}$(x)$ on $\mathcal{C}$. 
 If in addition $\mathcal{C}$ is unitary,  dim$=\mathrm{PFdim}$.
  
Given a fusion category $\mathcal{C}$, the \textit{(braided) centre} or \textit{(quantum) double construction}  associates to it
a modular tensor category $Z(\mathcal{C})$. The forgetful functor $Z(\mathcal{C})\rightarrow\mathcal{C}$
defines a ring homomorphism on the fusion rings and (hence) preserves dimensions. The forgetful functor has an adjoint
called the \textit{induction functor}.
Given a finite index and depth subfactor 
$N\subset M$, we obtain a (unitary) modular tensor category by applying the centre construction to its
principal even part. The modular tensor category associated to the dual even part will be equivalent,
but the two induction functors can carry independent information, as we'll see.

 Given a fusion category $\mathcal{C}$, or for that matter a subfactor $N\subset M$, it is very difficult
 to determine the centre $Z(\mathcal{C})$. A surprising discovery of Morrison--Walker is that it
 is often possible to determine a unique possibility for the induction functor at the level of the fusion
 rings. 
  
  Define a diagonal matrix $\tilde{T}$, with rows and columns indexed by $\Phi(\mathcal{C})$, by $\tilde{T}_{[x],[y]}=\delta_{[x],[y]} (\mathrm{tr}_x\otimes\mathrm{id}_x)(c_{x,x})$. Then $\tilde{T}$ is  well-defined and unitary. The assignment
  $s\mapsto \tilde{S}$, $t\mapsto \tilde{T}$ defines a projective representation of the modular group
  SL$(2,\bbZ)=\langle s,t\rangle$, where we put
  $s = \left(\begin{smallmatrix}0 & 1 \\ -1 & 0\end{smallmatrix}\right),
 t = \left(\begin{smallmatrix}1 & 1 \\ 0 & 1\end{smallmatrix}\right)$.
  The permutation matrix defined by $C_{[x],[y]}=\delta_{[y],[x^\vee ]}$ (where $x^\vee$ is the right or left dual of $x$) satisfies $C^2=I$ and commutes with
  both $\tilde{S}$ and $\tilde{T}$ --- it is often called \textit{charge-conjugation}. \textit{Verlinde's formula} computes the fusion coefficients of a 
  modular tensor category $\mathcal{C}$ in terms of $\tilde{S}$:
  \begin{equation}\label{Verlinde}N_{[x],[y]}^{[z]}=\mathcal{D}^{-2}\sum_{[w]\in\Phi(\mathcal{C})}\mathrm{dim}
  ([w])^{-1}\tilde{S}_{[x],[w]}\tilde{S}_{[y],[w]}\tilde{S}_{[z^\vee],[w]}\end{equation}
  where $\mathcal{D}^2=\sum_{[x]\in\Phi(\mathcal{C})}\mathrm{dim}([x])^2$.
  
\subsection{Modular data and congruence representations}
\label{sec:congruence-reps}

A modular tensor category is a fairly complicated beast. Remarkably, a highly constrained
combinatorial invariant of a modular tensor category seems close in practise to being a complete invariant.

\begin{Def} Let $\Phi$ be a finite set of labels, one of which (call it 1) is distinguished. By modular data
we mean matrices $S=(S_{xy})_{x,y\in\Phi}$, $T=(T_{xy})_{x,y\in\Phi}$ of complex numbers such that
\begin{enumerate}[(a)]
\item $S$ is unitary and symmetric; $T$ is unitary and diagonal;
\item $S_{1,x}\in\bbR^\times$ for all $x\in\Phi$; there is some $o\in\Phi$ such that $S_{o,x}>0$ for all $x\in\Phi$;
\item $S^2=(ST)^3$;
\item the numbers defined by
\begin{equation}N_{xy}^z=\sum_{w\in\Phi}\frac{S_{xw}S_{yw}\overline{S_{zw}}}{S_{1w}}\label{Verlinde2}
\end{equation}
are nonnegative integers, where the bar denotes complex conjugation.
\end{enumerate}
\end{Def} 

The matrices $\tilde{S},\tilde{T}$ coming from a modular tensor category can always be rescaled
so as to give modular data (with 1 being $[1]$) --- in particular, $S=\mathcal{D}^{-1}\tilde{S}$ ($\mathcal{D}$
is defined only up to a sign, but the sign should be chosen so that $\mathcal{D}^{-1}\tilde{S}$ has a 
strictly positive row). When the modular tensor category is unitary, $o$ in
(b) is also $[1]$. When the modular tensor category is the centre of a fusion category, then $T=\tilde{T}$.

The surprising lesson of this paper is that, although it is very difficult in general to obtain the modular tensor category from a fusion category or subfactor, it can be surprisingly easy to obtain the corresponding modular data.

There are several easy consequences of the definition of modular data. One is that it defines a 
(unitary)  SL$(2,\bbZ)$-representation $\rho$ through $s\mapsto S,t\mapsto T$. We will often call this
$\rho$ modular data. Also, $C=S^2$ 
is a permutation matrix $C_{x,y}=\delta_{y,x^\vee}$ commuting with $S$ and $T$, and satisfies
$C^2=I$ and
\begin{equation}\label{charge} \overline{S_{x,y}}= S_{x^\vee,y}\qquad\forall x,y\in\Phi\,.\end{equation}
Hence $1^\vee=1$ and $o^\vee=o$; moreover, $C=I$ iff $S$ is real.
The Perron--Frobenius dimensions are PFdim$(x)=\frac{S_{xo}}{S_{0o}}$.  When PFdim$(x)=1$,
then $x\,x^\vee=1$ in the fusion ring, and this has significant consequences for $S$ and $T$ (but
as we won't use these, we won't write them down).

The numbers $S_{xy}$ lie in some cyclotomic field $\bbQ[\xi_N]$. 
Then for each Galois automorphism $\sigma\in\mathrm{Gal}(\bbQ[\xi_N]/\bbQ)$, there is a permutation $x\mapsto x^\sigma$
of $\Phi$ and signs $\epsilon_\sigma:\Phi\rightarrow \{\pm1\}$ such that
\begin{equation}\label{galS}
  \sigma(S_{xy})=\epsilon_\sigma(x)S_{x^\sigma,y}=\epsilon_\sigma(y)S_{x,y^\sigma}\,.\end{equation}
  For example, complex conjugation corresponds to \eqref{charge}, i.e. to the permutation $x\mapsto x^\vee$
  and signs $\epsilon(x)=+1$.

 Verlinde's formula \eqref{Verlinde2} tells us the ratios $S_{xy}/S_{1y}$, being eigenvalues of the integer
 matrix $N_x=(N_{xa}^b)_{a,b\in\Phi}$, must be algebraic integers. Hence for any Galois automorphism
 $\sigma$, both $S_{1^\sigma,1}/S_{11}$ and 
 $$\frac{\epsilon_\sigma(1^{\sigma^{-1}})}{\epsilon_\sigma(1)}\sigma\left(\frac{S_{1^{\sigma^{-1}}1}}{S_{11}}
 \right)=\left(\frac{S_{1^\sigma 1}}{S_{11}}\right)^{-1}$$
are algebraic integers. But recall that dim$\,x=S_{x1}/S_{11}$ for any $x\in\Phi$. Thus we know that
dim$(1^\sigma)$ is an algebraic unit for all $\sigma$. This observation will help us identify later the Galois orbit of the unit $1$.

   Of course, Gal$(\bbQ[\xi_N]/\bbQ)\cong
\bbZ_N^\times$, where the correspondence $\sigma\leftrightarrow l$ is given by $\sigma(\xi_N)=\xi_N^l$.
We'll write $\sigma_l$ for the automorphism corresponding to $l\in\bbZ_N$. For example, complex 
conjugation is $\sigma_{-1}$.
We can say much more for the modular data associated to a modular tensor category.
 
We let
 $\Gamma(N)$ denote the principal congruence subgroup $$\{A\in \mathrm{SL}(2,\bbZ)\,|\,A\equiv
 I\ (\mathrm{mod}\ N)\}\,.$$ 
 We call $N$ the \textit{conductor}
of an SL$(2,\bbZ)$-representation $\rho$ if $N$ is the
smallest positive integer such that $\Gamma(N)$ is in the kernel of $\rho$ (and  $N=\infty$ if no $\Gamma(M)$
is in the kernel). We call $N$ the \textit{conductor} of a  field $K\supseteq \bbQ$ if $N$ is the smallest positive integer 
such that $K\subseteq \bbQ[\xi_N]$ (and $N=\infty$ if no cyclotomic field contains $K$).

 \begin{prop}[c.f. {\cite[Theorem 6.8]{MR2725181} \cite{math/9909080,MR1962117}}] \label{1.2} 
 Let $S,T,\rho$ be the modular data of a modular tensor category.
 Let $N$ be the order of $T$. Then $N<\infty$, $N$ equals the conductor of $\rho$, and
$N$ is a multiple of the  conductor of the field  $\bbQ[S]$ generated by all entries $S_{xy}$.
 Moreover, 
\begin{equation}
\label{galT}
T_{x^\sigma,x^\sigma}=T_{xx}^{l^2}\qquad\mathrm{for}\ \sigma=\sigma_l
\end{equation} 
for any  $\sigma\in\mathrm{Gal}(\bbQ[\xi_N]/\bbQ)$. If we define a signed
permutation matrix $G_\sigma$ by $(G_\sigma)_{x,y}=\epsilon_\sigma(x)\delta_{y,x^\sigma}$, then
 \begin{equation}\label{galoismatrix}
G_\sigma=CST^{1/l}ST^lST^{1/l}\qquad\mathrm{for}\ \sigma=\sigma_l\,,\end{equation}
where `$1/l$' denotes the inverse mod $N$ of $l$.
\end{prop}

Now, $\Gamma(N)$ is normal in SL$(2,\bbZ)$, with quotient SL$(2,\bbZ)/\Gamma(N)\cong \mathrm{SL}(2,\bbZ_N)$. Thus this fact tells us that
$\rho$ factors through to a representation of the finite group SL$(2,\bbZ_N)$, which we will also denote
by $\rho$. It also tells us that
$G_\sigma=\rho(\gamma)$, where $\gamma$ is any element in SL$(2,\bbZ)$ congruent mod $N$
to $ \left(\begin{smallmatrix}\ell & 0 \\ 0 & 1/\ell\end{smallmatrix}\right)$, where $\sigma=\sigma_l$.

  We write $\chi_i^{(N)}$ for the SL$(2,\bbZ_N)$-character denoted X.i by GAP, and denote by
  $\rho_i^{(N)}$ the corresponding representation. This labelling is generally not unique, and depends on how the
  conjugacy classes are identified with the columns of GAP's character table, but for the SL$(2,\bbZ_N)$ we need, we
will make this explicit. For example, for SL$(2,\bbZ_2)$ we assign the generators $S,T$ to class 2a, while for SL$(2,\bbZ_3)$ we assign  $S,T$ to class 4a and 3b, respectively. An SL$(2,\bbZ_N)$-irrep $\rho$ obeys $\rho(-I)=\pm I$. If it is $+I$ we call $\rho$ \textit{even},
in which case it factors through to an irrep of PSL$(2,\bbZ_N)$; if  $\rho(-I)=-I$, we call $\rho$
\textit{odd}.

Given a $d$-dimensional SL$(2,\bbZ)$-representation $\rho$, write $\mathcal{T}(\rho)$ for the multiset 
$\{t_1,\ldots,t_d\}$ where $\{e^{2\pi i t_j}\}$ is the list of eigenvalues of $\rho(t)$. For us, $\rho(t)$ will
always have finite order, so the $t_j\in\bbQ/\bbZ$.
One easy consequence of an
SL$(2,\bbZ)$-representation $\rho$ having finite conductor $N$ is that the multiset of $T^{l^2}$-eigenvalues 
is independent of $l\in\bbZ_N^\times$: \begin{equation}\label{galoiscong} 
\{t_1,\ldots,t_d\}
=\{l^2t_1,\ldots,l^2t_d\}\qquad\mathrm{for\ all}\ l\in\bbZ_N^\times\,.\end{equation}
To see this, note that in SL$(2,\bbZ_N)$, $t^{l^2}$ equals $t$ conjugated by $  \left(\begin{smallmatrix}\ell & 0 \\ 0 & 1/\ell\end{smallmatrix}\right)$, and so $T^{l^2}$ must have the same multiset of eigenvalues
as $T$.

Let $\prod_p p^{\nu_p}$ be the prime decomposition of $N$. 
By the Chinese Remainder Theorem, the group SL$(2,\bbZ_{N})$ is isomorphic to the direct product of
the SL$(2,\bbZ_{p^{\nu_p}})$. This implies that the irreps of SL$(2,\bbZ_N)$ are the tensor products 
$\otimes_p\rho_p$, where each $\rho_p$ is an irrep of  SL$(2,\bbZ_{p^{\nu_p}})$.

For example, the even 1-dimensional SL$(2,\bbZ)$-representations are $\rho_1^{(1)},\rho^{(3)}_2,\rho^{(3)}_3$,
while the odd ones are $\rho_2^{(2)},\rho_2^{(2)}\otimes\rho^{(3)}_2,\rho_2^{(2)}\otimes\rho^{(3)}_3$. These have
$$\cT(\rho)=\{0\},\left\{\frac{2}{3}\right\},\left\{\frac{1}{3}\right\},\left\{\frac{1}{2}\right\},\left\{\frac{1}{6}\right\},\left\{\frac{5}{6}\right\},$$ respectively. 

Finally we have Cauchy's theorem for modular tensor categories, recently proved in \cite{1310.7050}:
\begin{prop}
\label{fact:cauchy}
The primes dividing the conductor of a modular tensor category are the same primes that divide the norm of its global dimension.
\end{prop}
(This result is not actually essential to what follows. Our first derivation of the modular data did not use this, but it considerably simplifies the analysis.)

\section{Galois actions}
\begin{lem}
\label{lem:bounding-orbit-size-from-N}
Suppose a simple object $x$ has $T_{xx}$ a root of unity with order $N_x = \prod_p p^{\mu_p}$.
Then the number of distinct eigenvalues of $T$ in the full Galois
orbit
of $x$ is
\begin{align*}
k(N_x) & = \max\{1, 2^{\mu_2  - 3}\} \prod_{2 < p | N_x} p^{\mu_p - 1} (p-1)/2 \\ 
       & = N_x 2^{-\min\{\mu_2, 3\}} \prod_{2 < p | N_x} \frac{1}{2} \left(1-\frac{1}{p}\right).
\end{align*}
The size of the full Galois orbit is thus a multiple of $k(N_x)$.
\end{lem}
\begin{proof}
Suppose the order of $T$ is $N$, some multiple of $N_x$. For any $\ell \in \bbZ_{N_x}^\times$, there is an $\ell' \in \bbZ_N^\times$ with $\ell' \equiv \ell \pmod {N_x}$. Now
$$T_{\sigma_{\ell'}x \sigma_{\ell'}x} = T_{xx}^{{\ell'}^2} = T_{xx}^{\ell^2}.$$
Thus there are as many distinct eigenvalues of $T$ in the orbit of $x$ as there are images of the squaring map in $\bbZ_{N_x}^\times$.

Now, $\bbZ_{mn}^\times \cong \bbZ_m^\times\times \bbZ_n^{\times}$ if gcd$(m,n)=1$. Moreover, $\bbZ^\times_
{p^n} \cong
\bbZ_{p^{n-1}(p-1)}$ for prime $p\ne 2$ and any $n$, and
$\bbZ^\times_
{2^n} \cong \bbZ_2\times \bbZ_ {2^ {n-2}}$ for $n \geq 2$. Those
well-known facts give us the structure of any $\bbZ_N^\times$, and hence the cardinality of the image of the squaring map, as given above.
\end{proof}

\begin{cor}
\label{lem:conductor-squared}
Let $S,T$ be the modular data of some modular tensor category.
Suppose a prime $p$ divides the conductor of $\bbQ[d_x]$ for some $x\in\Phi$, and $\|\Phi\| < p (p-1)/2$. Then the order $N$ of $T$ is $p
M$, where $M$ is coprime to $p$.
\end{cor}
\begin{proof}
Since $d_x=S_{x1}/S_{11}$, $\bbQ[d_x]\subseteq \bbQ[S]$. 
Certainly $p$ divides $N$, by Lemma 1.2. If $p^2$ divides $N$, then there would be some $y\in\Phi$ 
with root of unity $T_{yy}$ having order $N_y$ a multiple of $p^2$. Then Lemma \ref{lem:bounding-orbit-size-from-N} would imply
$\|\Phi\|\ge k(N_y)\ge k(p^2)=p(p-1)/2$, a contradiction.
\end{proof}

\section{Basic lemmas}
\label{sec:basic-lemmas}

Let $S,T$ be the modular data coming from a modular tensor  category, and let $\rho$ be the 
corresponding SL$(2,\bbZ)$ representation. Write $N$ for the order of $T$ and $\Phi$ for the set of simple
objects. Recall the multiset $\cT(\rho)=\{t_x\}_{x\in\Phi}$ defined last section, so $T_{xx} = \exp(2\pi i t_x)$. As
always, $1\in\Phi$
denotes the unit and $x^\vee$ the dual.

In the following, we will assume for convenience that $T_{11}=1$, 
and that $S_{x1}>0$. Both are true for instance for the double of any subfactor (as categorical dimensions coincide with Frobenius-Perron dimensions, which are positive). All of
our results can be easily generalised when those assumptions are dropped.

Because $\rho$ is a representation of the finite group SL$(2,\bbZ_N)$, it decomposes into a direct sum
$\rho\cong\oplus_{i\in\cI}\,\rho_i$ of irreps.
Our strategy will be to control the possibilities for this decomposition.
Write $S_i=\rho_i(s)$ and $T_i=\rho_i(t)$. Like $\rho$, each $\rho_i$ is a matrix representation; bases
$\Phi_i$ are chosen so that each $T_i$ is diagonal.
Then there will exist an invertible matrix $Q$, with entries $Q_{iz,x}$ for $i\in\cI,z\in\Phi_i,x\in\Phi$,
 such that $S=Q^{-1}(\oplus_iS_i)Q$ and $T=Q^{-1}(\oplus_iT_i)Q$. Write $N=\prod_pp^{\nu_p}$
 as before; then $\rho_i\cong\otimes_p\rho_{i,p}$ where $\rho_{i,p}$ is some irrep of SL$(2,\bbZ_{p^\nu_p})$.

Call $i\in\cI$ \textit{even} resp. \textit{odd} if the subrepresentation $\rho_i$ is even resp. odd.
Call a simple object $x$ \textit{unique} if $t_{x}$ occurs with multiplicity one in $\cT(\rho)$. 

Let's collect some simple observations. (See \cite[\S 3]{1507.05139} for some related statements.)
\begin{lem}\label{2.1}\label{lem:basic}\mbox{}
\begin{enumerate}[(a)]
\item 
\label{unique-a} If $Q_{iz,x}\ne 0$ or $(Q^{-1})_{x,iz}\ne 0$, then $T_{i;\,zz}=T_{xx}$.

\item
\label{unique-b}
Suppose $S_{xy}\ne 0$. Then  both $t_{x},t_{y}\in\cT(\rho_i)$ for some index $i\in\cI$.

\item
\label{unique-c}
For each $x\in\Phi$, write $t_x=\sum_p \frac{m_p}{p^{\nu_p}}$ for $m_p\in\bbZ$; then
there is a (not necessarily unique) index $i_x\in\cI$ such that both $0,m_p/p^{\nu_p}\in\cT(\rho_{i_x,p})$ 
for all $p$.

\item
\label{unique-d}
Suppose $x\in\Phi$ is unique. Then $i_x$ defined in \eqref{unique-c} is unique. Let $\hat{x}$ denote the unique
index in  $\Phi_{i_x}$ with $T_{i_x;\,\hat{x}\hat{x}}=T_{xx}$. Then for all $i\in\cI,z\in\Phi_i,y\in\Phi$,
$Q_{iz,x}=\cQ_x\delta_{ii_x}\delta_{z\hat{x}}$ and $Q_{i_x\hat{x},y}=\cQ_x \delta_{xy}$,
for some nonzero $\cQ_x$.  

\item
\label{unique-e}
Suppose $x,y\in\Phi$ are both unique and that $i_x=i_y$.  Then $S_{i_x;\,\hat{x}\hat{y}}\cQ_y^2=S_{i_x;\,\hat{y}\hat
{x}}\cQ_x^2$
and $S_{xy}=S_{i_x;\,\hat{x}\hat{y}}\cQ_y/\cQ_x$.

\item
\label{unique-f}
Suppose $x\in\Phi$ is unique and $0\in\cT(\rho_{i_x})$ has multiplicity one.  Write $z_x$ for the 
 unique index in $\Phi_{i_x}$ with $T_{i_x;z_xz_x}=1$. Then for any  $y\in\Phi$ with
$T_{yy}=1$, $S_{xy}=S_{i_x;\,\hat{x}z_x}Q_{z_x,y}/\cQ_x=(Q^{-1})_{y,i_xz_x}\cQ_xS_{i_x;\,z_x\hat{x}}$.

\item 
\label{unique-g}
For each $r\in\cT(\rho)$, let $$n_+(r)=\sum_{\mathrm{even}\,\, i\in\cI}\mathrm{mult}_{\cT(\rho_i)}(r)$$ and 
$$n_-(r)=\sum_{\mathrm{odd}\,\, i\in\cI}\mathrm{mult}_{\cT(\rho_i)}(r).$$ Then $n_+(r)+n_-(r)=\mathrm{mult}_{\cT(\rho)}(r)$ and $n_+(r)-n_-(r)$ is the number of $x=x^\vee\in\Phi$ with $t_x=r$. In particular,
$n_+(r)\ge n_-(r)$.
\end{enumerate}
\end{lem}

\begin{proof}
Because both $T$ and $\oplus_iT_i$ are diagonal, the $(iz,x)$-entries of $QT=(\oplus_iT_i)Q$ and
$TQ^{-1}=Q^{-1}(\oplus_iT_i)$ give \eqref{unique-a}. To see \eqref{unique-b}, 
suppose $S_{xy}\ne 0$. Since $S_{xy}=\sum_{i,a,b}(Q^{-1})_{x,ia}S_{i;\,ab}Q_{ib,y}$, this means
there
is some indices $i\in\cI$ and  $a,b\in\Phi_i$ such that both
$(Q^{-1})_{x,ia},Q_{ib,y}\ne 0$. From \eqref{unique-a}, this gives \eqref{unique-b}. Part \eqref{unique-c}
now follows from \eqref{unique-b} and $S_{x1}\ne 0$:  $\cT(\rho_{i;p})\subset p^{-\nu_p}\bbZ/\bbZ$
and any $r\in\cT(\rho)$ will  have a unique (mod 1) expression as a sum $\sum_pm_p/p^{\nu_p}$. 
Part \eqref{unique-d} is immediate from \eqref{unique-a}. Parts \eqref{unique-e} and \eqref{unique-f} now follow  from $S_{yx}=S_{xy}=(Q^{-1}\left(\oplus_iS_i\right)Q)_{xy}$.

To see part \eqref{unique-g}, restrict charge-conjugation $S^2=Q^{-1}\left(\oplus_iS^2_i\right)Q$ to the $x\in\Phi$ with $t_x=r$. The trace 
of that permutation submatrix will equal the number of self-dual $x$ with $t_x=r$; since $S_i^2=\pm I$ depending
on whether  $\rho_i$ is even or odd, that trace will also equal $n_+(r)-n_-(r)$.
\end{proof}


\section{Dimensions}

In \cite{1404.3955}, the combinatorial data $A: K_0(Z(\cC)) \to K_0(\cC)$ of the restriction functor  $Z(\cC)\rightarrow
\cC$
was obtained,
when $\cC$ is both the principal even and dual even fusion categories of the extended Haagerup.
These are respectively
$$A_{EH1}=\left(\begin{array}{cccccccccccccccccccccc} 1&  1&  1&  1&  1&  1&  0& 0&  0&  0&  0&  0&  0&  0& 0&  0&  0& 
0&  0&  0&    0&  0\\ 0&  1&  1&  2&  1&  0&  1& 1&  1&  1&  1&  1&  1&  1&  0&  0&  0&  0&  1&  1&   1&  1\\ 0&  2&  1&  1&  1&  3&  2&  2&  2&  2&  2&  2&  2&  2&  1&  1&  1&  1&  1&  1&   1&  1\\ 0&  4&  1&  2&  4&  2&  3&   3&  3&  3&  4&  4&  4&  4&  3&  3&  3&  3&  1&  1&   1&  1\\ 0&  5&  1&  4&  2&  3&  3&   3&  3&  3&  5&  5&  5&  5&   4&  4&  4&  4&  1&  1&   1&  1\\ 0&  3&  1&  1&  2&  2&  1&   1&  1&  1&  3&  3&  3&  3&   2&  2&  2&  2&  1&  1&    1&  1\end{array}\right)$$
$$A_{EH2}=\left(\begin{array}{cccccccccccccccccccccc} 1& 1& 1& 2& 1& 0& 0& 0& 0& 0& 0& 0& 0& 0&  0& 0& 0& 0& 0& 0& 0&
0\\ 0& 1& 1& 1& 1& 1& 1& 1& 1& 1& 1& 1& 1& 1&  0& 0& 0& 0& 1& 1& 1& 1\\ 0& 2& 1& 2& 1& 2& 2& 2& 2& 2& 2& 2& 2& 2& 1& 1& 1& 1& 1& 1& 1& 1\\ 0& 4& 1& 1& 4& 3& 3& 3& 3& 3& 4& 4& 4& 4& 3& 3& 3& 3& 1& 1& 1& 1\\ 0& 4& 1& 3& 2& 2& 2& 2& 2& 2& 4& 4& 4& 4& 3& 3& 3& 3& 1& 1& 1& 1\\ 0& 4& 1& 3& 2& 2& 2& 2& 2& 2& 4& 4& 4& 4& 3& 3& 3& 3& 1& 1& 1& 1\\ 0& 1& 0& 1& 0& 1& 1& 1& 1& 1& 1& 1& 1& 1&  1& 1& 1& 1& 0& 0& 0& 0\\ 0& 1& 0& 1& 0& 1& 1& 1& 1& 1& 1& 1& 1& 1& 1& 1& 1& 1& 0& 0& 0& 0\end{array}\right).$$
The matrices corresponding to the induction functors are the transposes.
The 22 columns correspond to the 22 simple objects $\Phi$ in the centre $Z(\cC)$. The columns have been 
ordered so that the first column corresponds to the tensor unit. For reasons which will be clear shortly, 
we will name these 22 simple objects, in order, $\omega_0,\omega_1,\omega_2, \alpha_1,\alpha_2,\alpha_3,\beta_1,\ldots,\beta_4,\gamma_1,\ldots,\gamma_4,\delta_1,\ldots,\delta_4,\epsilon_1,\ldots,\epsilon_4$. Here $\omega_0$ is the tensor identity.

Each restriction matrix tells
us two things. First, the image of the tensor identity in $\cC$ (namely the first row in $A_\cC$) will be an eigenvector
of both $S$ and $T$, with
eigenvalue 1 and $T_{11}$ respectively \cite[Theorem 1]{MR2837122}. As in any centre, we can take $T_{11}=1$ here; this tells us for instance that 
\begin{equation}T_{\omega_i\omega_i}=1=T_{\alpha_j\alpha_j}\end{equation}
for all $0\le i\le 2$ and $1\le j\le 3$.
Second and far more important, we obtain the dimensions dim$\,x$ for the simple $x$
in $Z(\cC)$: these dimensions are the components of the Perron--Frobenius eigenvector of the
matrix $A_\cC^tA_\cC$, normalised so that dim$\,1=1$. Its  eigenvalue will be the global dimension
$\cD=\sqrt{\sum_{x\in\Phi}\mathrm{dim}(x)^2}$. Of course, $S_{1x}=S_{x1}=\frac{\dim(x)}{\cD}$.

Numerically, these dimensions are approximately 1,  177.701,  49.396,  
   114.049 (7 times),  176.701 (4 times),  128.304 (4 times),  
and     48.396 (4 times), respectively, and the global dimension $\cD$ is approximately 570.246. 
When they are computed exactly, they are all found to lie in  the degree-3 extension $\bbQ_{dim}$ of $\bbQ$ in $\bbQ[\xi_{13}]$. 
More precisely, $\bbQ_{dim}$ has a basis $1,\zeta=2\cos(2\pi/13)+2\cos(10\pi/13)$ and 
$\zeta'=2\cos(4\pi/13)+2\cos(6\pi/13)$ over $\bbQ$; then
\begin{align}
\label{eq:dimensions}
S_{11} & = \frac{7-5\zeta'}{65}, &
S_{1,\omega_1} & = \frac{12+ 5\zeta+5\zeta'}{65}, &
S_{1,\omega_2} & = \frac{7- 5\zeta}{65}, \\
 & &  S_{1,\alpha_i} & = S_{1,\beta_j} = \frac{1}{5}, \notag \\
S_{1,\gamma_j} &=
\frac{1+\zeta+2\zeta'}{13}, & S_{1,\delta_j} & =\frac{1+ 2\zeta+ \zeta'}{13}, & S_{1,\epsilon_j}&=\frac{-\zeta+\zeta'}{13}. \notag
\end{align}

\section{Galois action and the conductor}
\label{sec:galois-action}

\begin{thm}\label{compatible}
Any modular data compatible with the restriction matrices given in the last section has
\begin{enumerate}
\item conductor $N = 5 \times 13$,
\item first 3 rows and columns of $S$ determined by
$$ S _{\omega_i, x} =  S _{x,\omega_i} =\sigma_{16}^i S_{1, x}, $$
where $\sigma_l(\xi_{65}) = \xi_{65}^l$,
and
\item 
\begin{enumerate}[(i)]
\item the objects $\{\omega_i\}$ forming a single Galois orbit,
\item the objects $\{\alpha_1,\ldots,\alpha_3,\beta_1,\ldots,\beta_4\}$ forming a union of Galois orbits, and
\item the objects $\{\gamma_1,\ldots,\gamma_4, \delta_1,\ldots,\delta_4, \epsilon_1,\ldots,\epsilon_4\}$ either forming
a single Galois orbit of size 12, or forming two Galois orbits of size 6, each of which containing two each of the $\gamma_i$, $\delta_i$, and $\epsilon_i$.
\end{enumerate}
\end{enumerate}
\end{thm}
\begin{proof}
Define $\bar{\sigma}_l \in \Gal(\bbQ[\xi_{13}] / \bbQ)$ by $\bar{\sigma}_l(\xi_{13}) = \xi_{13}^l$. We see that $\Gal(\bbQ_{dim}/\bbQ)=\{\bar{\sigma}_1,\bar{\sigma}_3,\bar{\sigma}_{9}\}$.

Then $\cD=295+ 125\zeta+ 175\zeta'$,  which has norm $\cD\bar{\sigma}_3(\cD)\bar{\sigma}_3^2(\cD)=21125=5^3 13^2$. By Cauchy's theorem for modular tensor categories (Fact \ref{fact:cauchy}), the order $N$ of $T$ will be $5^a13^b$ for some $a,b\ge 1$.
By Corollary \ref{lem:conductor-squared}, $b=1$.

Whatever the value of the conductor, we have a surjective map $\pi:\bbZ^\times_N\rightarrow
  \{\bar{\sigma}_1,\bar{\sigma}_3,\bar{\sigma}_9\}$, corresponding to the restriction of $\sigma\in
  \Gal(\bbQ[\xi_N]/\bbQ)$ to $\bbQ_{dim}$, which we'll write $\pi\sigma_l=\bar{\sigma}^{\pi l}_3$.

From \eqref{galS} and $S_{1x}>0$, we obtain the sign $\varepsilon_l(x)=\mathrm{sign}(\sigma_l S_{1,x})$ for any 
$l\in\bbZ^\times_N$ and any $x\in\Phi$.
Since $\bar{\sigma}_3(S_{1,1})=S_{1,\omega_1}$ and $\bar{\sigma}_3^2(S_{1,1})=S_{1,\omega_2}$, $\bar{\sigma}_3(S_{1,\alpha_1})=S_{1,\alpha_1}$, $\bar{\sigma}_3(S_{1,\gamma_1})=-S_{1,\delta_1}$, $\bar{\sigma}_3(S_{1,\delta_1})=S_{1,\epsilon_1}$,  we obtain
$$\varepsilon_l(\omega_i)=\varepsilon_l(\alpha_i)=\varepsilon_l(\beta_j)=\varepsilon_{\bar{\sigma}_3}(\delta_j)=+1\,,\
\varepsilon_{\bar{\sigma}_3}(\gamma_j)=\varepsilon_{\bar{\sigma}_3}(\epsilon_j)=-1\,,\ \forall i,l,j\,.$$
Moreover, $\bbZ^\times_N$ sends $\{\alpha_1,\ldots,\alpha_3,\beta_1,\ldots,\beta_4\}$ to itself, and
$$\omega_i^{\sigma_l}=\omega_{i+\pi(l)}\ \ \forall i,l\,.$$ 
When $\pi(l)=1$, $\sigma_l$ sends $\{\gamma_1,\ldots,\gamma_4\}\rightarrow\{\delta_1,\ldots,\delta_4\}
\rightarrow\{\epsilon_1,\ldots,\epsilon_4\}\rightarrow \{\gamma_1,\ldots,\gamma_4\}$. 

Using this Galois action, we obtain $S_{\omega_i,x}=\bar{\sigma}_3^iS_{1,x}$ for $i=1,2$. Thus we
know the first 3 rows and columns of $S$ (as well as the first 6 diagonal elements of $T$, of course).

Because we know $N=5^a13$, there will exist a unique order-3 element $l$ in $\bbZ^\times_N$ 
with $\pi(l)=1$, namely $l\equiv 3$ (mod 13) and $l\equiv 1$ (mod $5^a$). We will also use 
$\bar{\sigma}_3 $ to denote this element of $\Gal(\bbQ[\xi_N]/\bbQ)$.
As the elements of the sets $\{\gamma_1,\ldots,\gamma_4\}$, $\{\delta_1,\ldots,\delta_4\}$, and
$\{\epsilon_1,\ldots,\epsilon_4\}$ are at this point indistinguishable, we
may choose $\bar{\sigma}_3(\gamma_i) = \delta_i$, $\bar{\sigma}_3(\delta_i) = \epsilon_i$, and $\bar{\sigma}_3(\epsilon_i) = \gamma_i$.

Thus we see that the objects $\{\gamma_1,\ldots,\gamma_4, \delta_1,\ldots,\delta_4, \epsilon_1,\ldots,\epsilon_4\}$ form between one and four Galois orbits, with these orbits having size a multiple of 3.
But \eqref{galT} implies that the length of this Galois orbit must be even if $T_{\gamma_i,\gamma_i}$
has order a multiple of 5, and the length  must be a multiple of 6 if the order is a multiple of 13.
This gives us (iii).

Suppose now that $5^2$ divides $N$. Then there is a simple object $x$ with $T_{xx}$ a root of unity with order $N_x$ divisible by 25. By Lemma \ref{lem:bounding-orbit-size-from-N}, the Galois orbit containing $x$ has size a multiple of 10. From the above, this is impossible. Thus we have proved that $N = 5 \times 13$.

Finally we see that $\bar{\sigma}_3$ is $\sigma_{16} \in \Gal(\bbQ[\xi_{65}] / \bbQ)$.
\end{proof}
\section{The group of 12}
\label{sec:group-of-12}

The character table of SL$(2,\bbZ_{13})$ (computed from GAP) is given in Figure 1. 
The number $\overline{A}=(1-\sqrt{13})/2$, so labelled because it is a Galois associate of $A=(1+\sqrt{13})/2$.
Class 2a is the central element, $s$ and $t$ correspond to class 4a and 13a respectively, while 12a generates the Galois group $\bbZ_{13}^\times$.

\begin{figure}[ht]
\label{fig:sl(2,13)}
\[
\begin{array}{l|rrrrrrrrrrrrrrrrr}
 & 1a & 26a & 26b & 2a & 13a & 13b & 14a & 7a & 7b & 7c & 14b & 14c & 12a & 3a & 4a & 6a & 12b \\ 
\hline
\chi_{1} & 1 & 1 & 1 & 1 & 1 & 1 & 1 & 1 & 1 & 1 & 1 & 1 & 1 & 1 & 1 & 1 & 1 \\ 
\chi_{2} & 6 & A & \overline{A} & -6 & -A & -\overline{A} & 1 & -1 & -1 & -1 & 1 & 1 & . & . & . & . & . \\ 
\chi_{3} & 6 & \overline{A} & A & -6 & -\overline{A} & -A & 1 & -1 & -1 & -1 & 1 & 1 & . & . & . & . & . \\ 
\chi_{4} & 7 & A & \overline{A} & 7 & A & \overline{A} & . & . & . & . & . & . & -1 & 1 & -1 & 1 & -1 \\ 
\chi_{5} & 7 & \overline{A} & A & 7 & \overline{A} & A & . & . & . & . & . & . & -1 & 1 & -1 & 1 & -1 \\ 
\chi_{6} & 12 & -1 & -1 & 12 & -1 & -1 & B & D & C & B & D & C & . & . & . & . & . \\ 
\chi_{7} & 12 & -1 & -1 & 12 & -1 & -1 & C & B & D & C & B & D & . & . & . & . & . \\ 
\chi_{8} & 12 & -1 & -1 & 12 & -1 & -1 & D & C & B & D & C & B & . & . & . & . & . \\ 
\chi_{9} & 12 & 1 & 1 & -12 & -1 & -1 & -B & D & C & B & -D & -C & . & . & . & . & . \\ 
\chi_{10} & 12 & 1 & 1 & -12 & -1 & -1 & -C & B & D & C & -B & -D & . & . & . & . & . \\ 
\chi_{11} & 12 & 1 & 1 & -12 & -1 & -1 & -D & C & B & D & -C & -B & . & . & . & . & . \\ 
\chi_{12} & 13 & . & . & 13 & . & . & -1 & -1 & -1 & -1 & -1 & -1 & 1 & 1 & 1 & 1 & 1 \\ 
\chi_{13} & 14 & 1 & 1 & 14 & 1 & 1 & . & . & . & . & . & . & 1 & -1 & -2 & -1 & 1 \\ 
\chi_{14} & 14 & 1 & 1 & 14 & 1 & 1 & . & . & . & . & . & . & -1 & -1 & 2 & -1 & -1 \\ 
\chi_{15} & 14 & -1 & -1 & -14 & 1 & 1 & . & . & . & . & . & . & . & 2 & . & -2 & . \\ 
\chi_{16} & 14 & -1 & -1 & -14 & 1 & 1 & . & . & . & . & . & . & E & -1 & . & 1 & -E \\ 
\chi_{17} & 14 & -1 & -1 & -14 & 1 & 1 & . & . & . & . & . & . & -E & -1 & . & 1 & E 
\end{array}\]
\begin{align*}
A & = -\zeta_{13}^{2}-\zeta_{13}^{5}-\zeta_{13}^{6}-\zeta_{13}^{7}-\zeta_{13}^{8}-\zeta_{13}^{11} \displaybreak[1] \\
B & = -\zeta_{7}-\zeta_{7}^{6} \displaybreak[1] \\
C & = -\zeta_{7}^{3}-\zeta_{7}^{4} \displaybreak[1] \\
D & = -\zeta_{7}^{2}-\zeta_{7}^{5} \displaybreak[1] \\
E & = -\zeta_{12}^{7}+\zeta_{12}^{11} \displaybreak[1]
\end{align*}%
\newsavebox{\smlmat}
\savebox{\smlmat}{$\left(\begin{smallmatrix} \ell&0\\ 0&\ell^{-1} \end{smallmatrix}\right)$}%
\caption{The character table of \(SL(2,\bbZ_{13})\). 
}
\end{figure}

\begin{prop}
\label{prop:group-of-12} 
Let $\rho$ be the SL$(2,\bbZ)$-representation $\rho$ coming from the modular data of the centre of the extended
Haagerup. Then $\rho\cong \rho^{(13)}_{14}\oplus\rho_{(5)}$, where $\rho_{(5)}$ is some representation whose kernel
contains $\Gamma(5)$.
\end{prop}

\begin{proof} 
We learned in Theorem \ref{compatible} that the full Galois group leaves invariant  the sets $\{\omega_i\}$, $\{\alpha_i\}\cup\{\beta_i\}$,
 and $\{\gamma_i\}\cup\{\delta_i\}\cup\{\epsilon_i\}$ of simples. We also know that the order  of $T$ is
 $N=5 \times 13$.

Consider $\Phi_{13}$, the set of those simples $x$ whose $T_{xx}$
has order a multiple of 13. 
Because of Equation \eqref{galT},
the set $\Phi_{13}$  is a union of Galois orbits. By Lemma \ref{lem:bounding-orbit-size-from-N}, each such orbit has size divisible by $\frac{13-1}{2}=6$. The set $\Phi_{13}$ cannot contain an $\alpha_i$ or $\omega_i$ (because their $T$ is 1),
nor $\beta_i$ (because those either have $T=1$ or form Galois orbits of cardinality $\le 4$). So we have 
$\Phi_{13}\subseteq\{\gamma_i\}\cup\{\delta_i\}\cup\{\epsilon_i\}$.

From the character table we find that
the only nontrivial irreps $\rho'$  of SL$(2,\bbZ_{13})$ for which  $0\in\mathcal{T}(\rho')$ are
$\rho^{(13)}_4$,  $\rho^{(13)}_5$, and the irreps $\rho_{12}^{(13)}$ to $\rho^{(13)}_{17}$.

First, suppose for contradiction that $\rho$ contains a subrepresentation of the form $\rho_{13}\otimes
\rho_5$, where $\rho_{13}$ resp. $\rho_5$ are irreps with conductor exactly 13 resp. 5. If $\rho_5$ has dimension
at least 3, then $(\rho_{13}\otimes\rho_{5})(t)$ will have at least $6\times 3=18$ diagonal
entries with order a multiple of 13, contradicting $\|\Phi_{13}\|\le 12$. Hence $\rho_5$ has dimension 2,
so by the same argument all $\Phi_{13}$ is accounted for by $\rho_{13}\otimes\rho_5$, and any
other subrepresentation of $\rho$ must have conductor coprime to 13 (and hence dividing 5). But
dim$\,\rho_5=2$ implies $0\not\in\mathcal{T}(\rho_5)$ thanks to Equation \eqref{galT}, and this contradicts
Lemma \ref{2.1}(c).

Hence $\rho\cong\rho_{13}\oplus\rho_{(5)}$, where every subrepresentation of $\rho_{13}$ has conductor exactly 13, and every subrepresentation of $\rho_{(5)}$ has conductor coprime to 13 (hence dividing 5).
Moreover, we know by Lemma \ref{2.1}(c) that $0\in\cT(\rho_{13})$.
We will constrain $\rho_{13}$ by considering the Galois matrix $G_{11}=\rho\left(\begin{smallmatrix}
11&0\\ 0&11^{-1}\end{smallmatrix}\right)$, which we know from Proposition \ref{1.2} is a signed permutation
matrix. This permutation $x\mapsto x^{\sigma_{11}}$ permutes $\Phi_{13}$ without fixed points, since $\bar{\sigma}_3=\sigma_{11}^8$
acts without fixed points. Likewise, $\sigma_{11}$ permutes $\omega_0,\omega_1,\omega_2$ without
fixed points, since $\bar{\sigma}_3$ does. Therefore $\sigma_{11}$ leaves invariant the sets $\{\alpha_i\}
\cup\{\beta_j\}$, as well as that part of $\{\gamma_i\}\cup\{\delta_i\}\cup\{\epsilon_i\}$ not in $\Phi_{13}$.
Of course, $\rho_{(5)}\left(\begin{smallmatrix}
11&0\\ 0&11^{-1}\end{smallmatrix}\right)=I$ since $\rho_{(5)}$ has conductor dividing 5. Together, this
means  $\mathrm{dim}\,\rho_{(5)}+\chi_{13}(12a)=\mathrm{Tr}\,G_{11}$ is the trace of a signed permutation
matrix with $22-\|\Phi_{13}\|-3$ rows, i.e.
\begin{equation}\mathrm{dim}\,\rho_{13}-\|\Phi_{13}\|-\chi_{13}(12a)\in\{3,5,6,7,\ldots\}\,.\label{charconstr13}
\end{equation}
Here, `$12a$' refers to the conjugacy class of $\left(\begin{smallmatrix}
11&0\\ 0&11^{-1}\end{smallmatrix}\right)$; the value 4 is excluded because the trace of a signed
permutation matrix of size $n\times n$ cannot equal $n-1$ (nor be larger than $n$).

Suppose next for contradiction that $\|\Phi_{13}\|<12$. Then $\|\Phi_{13}\|=6$,  and $\rho_{13}$ is $\rho_4^{(13)}$ or $\rho_5^{(13)}$.  In this case, dim$\,\rho_{13}-\|\Phi_{13}\|-\chi_{13}(12a)=2$, a forbidden value.
Similarly, if  $\|\Phi_{13}\|=12$ but $\rho_{13}$ is not irreducible, then $\rho_{13}\cong\rho'\oplus\rho''$, 
where $\rho'\in\{\rho_4^{(13)},\rho_5^{(13)}\}$ and $\rho''\in\{\rho^{(13)}_2,\rho^{(13)}_3,\rho^{(13)}_4,\rho^{(13)}_5\}$. But then dim$\,\rho_{13}-\|\Phi_{13}\|-\chi_{13}(12a)$ equals 2 (if $\rho''\in\{\rho^{(13)}_2,\rho^{(13)}_3\}$) or 4 (if $\rho''\in\{\rho^{(13)}_4,\rho^{(13)}_5\}$), both of which are forbidden.

Thus $\rho_{13}$ is irreducible and of dimension $\ge 13$ (since $0\in\cT(\rho_{13})$), 
so $\rho_{13}$ is one of $\rho^{(13)}_{12},\ldots,\rho^{(13)}_{17}$. 
We can dismiss $\rho_{i}\cong \rho^{(13)}_{ 15},\rho^{(13)}_{16},\rho^{(13)}_{17}$ out of hand, because these are odd, contradicting Lemma \ref{2.1}(g). Moreover, $\rho_{13}\cong\rho^{(13)}_{12}$ resp.
$\rho^{(13)}_{13}$ have dim$\,\rho_{13}-\|\Phi_{13}\|-\chi_{13}(12a)$ equal to 0 resp. 1, so also must
be dismissed. The only remaining possibility is $\rho_{13}\cong\rho^{(13)}_{14}$.
\end{proof}

We give an explicit matrix realisation of  $\rho_{14}^{(13)}$  in
Appendix \ref{appendix:matrices}.

\section{The group of 4}
\label{sec:group-of-4}

So far, we have accounted for the 12 simples $\{\gamma_i\}\cup\{\delta_i\}\cup\{\epsilon_i\}$, as well
as  4 simples $x$ with $T_{xx}=1$: namely  2 appearing in the $\rho_{14}^{(13)}$ (recall
Proposition
\ref{prop:group-of-12}), and 2 trivial SL$(2,\bbZ)$-irreps associated with the two modular invariants 
\begin{equation}\label{twoModInv}
(1\,1\,1\,1\,1\,1\,0\,\ldots\,0)^\transpose\,,\ \ (1\,1\,1\,2\,1\,0\,\ldots\,0)^\transpose\,,
\end{equation}
coming from the induction functors. That leaves  unaccounted  6 simples (amongst $\{\omega_i\}
\cup\{\alpha_i\}\cup
\{\beta_i\}=:\mathcal{R}$). We also know $\rho\cong\rho_{14}^{(13)}\oplus
\rho_{(5)}$, where  $\rho_{(5)}$ has conductor exactly 5. Our goal in this section is to identify $\rho_{(5)}$.

In Figure 2 we give the character table of SL$(2,\bbZ_5)$ (computed in GAP). Class $5a$ contains $t$,
class $4a$ contains both $s$ and $\left(\begin{smallmatrix}2&0\\ 0&2^{-1}\end{smallmatrix}\right)$, while class $2a$ contains $-I$. The number $\overline{A}=-2\cos(4\pi/5)$ is the unique nontrivial
Galois associate of ${A}=-2\cos(2\pi/5)$.

\begin{figure}[ht]
\label{fig:sl(2,5)}
\[
\begin{array}{l|rrrrrrrrr}
 & 1a & 10a & 10b & 2a & 5a & 5b & 3a & 6a & 4a \\ 
\hline
\chi_{1} & 1 & 1 & 1 & 1 & 1 & 1 & 1 & 1 & 1 \\ 
\chi_{2} & 2 & A & \overline{A} & -2 & -A & -\overline{A} & -1 & 1 & . \\ 
\chi_{3} & 2 & \overline{A} & A & -2 & -\overline{A} & -A & -1 & 1 & . \\ 
\chi_{4} & 3 & \overline{A} & A & 3 & \overline{A} & A & . & . & -1 \\ 
\chi_{5} & 3 & A & \overline{A} & 3 & A & \overline{A} & . & . & -1 \\ 
\chi_{6} & 4 & -1 & -1 & 4 & -1 & -1 & 1 & 1 & . \\ 
\chi_{7} & 4 & 1 & 1 & -4 & -1 & -1 & 1 & -1 & . \\ 
\chi_{8} & 5 & . & . & 5 & . & . & -1 & -1 & 1 \\ 
\chi_{9} & 6 & -1 & -1 & -6 & 1 & 1 & . & . & . \\ 
\end{array}\]
\begin{align*}
A & = -\zeta_{5}-\zeta_{5}^{4} \displaybreak[1] \\
\end{align*}
\caption{The character table of \(SL(2,\bbZ_{5})\).}
\end{figure}

\begin{prop}
\label{prop:group-of-4}
Let $\rho_{(5)}$ be as in Proposition \ref{prop:group-of-12}. Then $\rho_{(5)}\cong \rho_8^{(5)}\oplus
1\oplus 1\oplus 1$.  \end{prop}

\begin{proof} We can write $\rho_{(5)}=\rho_5\oplus\rho_1$, where every subrepresentation of $\rho_5$
has conductor exactly 5, and $\rho_1$ consists of exactly $(8-\mathrm{dim}\,\rho_5)$ copies of the trivial
representation 1.
The only $T_{xx}$ we need to constrain are the four $\beta_i$, because the other entries in
$\mathcal{R}$  all have $T_{xx}=1$. Recall from Theorem \ref{compatible} that $N=5\times 13$.
Let $\Phi_5$ consist of those $x\in\{\beta_1,\ldots,\beta_4\}$ with $t_x\ne 0$. Then $t_x\in\frac{1}{5}\bbZ$
for all $x\in\Phi_5$, and any $x\not\in\Phi_5\cup\Phi_{13}$ has $t_x=0$.

Suppose for contradiction that $\|\Phi_5\|<4$. Then $\rho_5\cong \rho_4^{(5)}$ or $\rho_5^{(5)}$, since
by Lemma \ref{2.1}(c) $0\in\cT(\rho_5)$. Suppose $\rho_5\cong\rho_4^{(5)}$ (the argument handling its
Galois associate $\rho_5^{(5)}$ is identical). The irrep $\rho^{(5)}_4$ is generated by matrices
$$S^{(5)}_4=\frac{1}{5}\left(\begin{matrix}c-c'&\sqrt{2}c-\sqrt{2}c'&\sqrt{2}c-\sqrt{2}c'\\ \sqrt{2}c-\sqrt{2}c'&2c+3c'&-3c-2c'\\ \sqrt{2}c-\sqrt{2}c'&-3c-2c'&2c+3c'\end{matrix}\right)\,,\ T^{(5)}_4=\mathrm{diag}(1,\xi_5,\xi_5^4)\,,$$
where $c=2\cos(2\pi/5)$, $c'=2\cos(4\pi/5)$. These have Galois matrix (recall \eqref{galoismatrix}) 
$$G_{2;4}^{(5)}=ST^3ST^2ST^3=\left(\begin{smallmatrix} -1&0&0\\ 0&0&-1\\ 0&-1&0\end{smallmatrix}\right)\,.$$
Write $x,x'$ for the unique simples with
$T_{xx}=\xi_5,T_{x'x'}=\xi_5^4$. Then by Lemma \ref{2.1}(e),(f) and 
$S_{1x}=S_{1x'}>0$, we would obtain $\cQ_x=\cQ_{x'}$ and $\varepsilon_2(x)=(G_2)_{xx'}=-\cQ_{x'}/\cQ_x=-1$, contradicting that we know  $\epsilon_\sigma(\beta_i)= +1$ for all Galois automorphisms $\sigma$. 

Therefore, $\|\Phi_5\|=4$, so there are exactly 6 simples $x\in\Phi$ with $T_{xx}=1$, namely $\{\omega_i\}
\cup\{\alpha_i\}$. From \eqref{galT}, the Galois automorphism $\sigma_{-12}$ fixes each $x\in\Phi_{13}$,
and permutes $\Phi_5$ without fixed points. From Theorem \ref{compatible}(3)(i) and the values 
$S_{1,\omega_i}\in\bbQ[\xi_{13}]$, we know $\sigma_{-12}$ fixes each $\omega_i$. The modular
invariant $(1\,1\,1\,2\,1\,0\,0^{16})^\transpose$   must by definition be an eigenvector of all $\rho(\gamma)$,
and hence $\rho\left(\begin{smallmatrix}-12&0\\ 0&-12^{-1}\end{smallmatrix}\right)$, with eigenvalue 1,
which implies $\sigma_{-12}$ fixes each $\alpha_i$.  We already knew all $\varepsilon_{-12}
(x)=+1$. Therefore, exactly as in the derivation of \eqref{charconstr13}, we obtain Tr$\,G_{-12}=
22-4=\chi_{13}(1a)+\chi_5(4a)+\chi_1(1a)$, i.e.
\begin{equation} \mathrm{dim}\,\rho_5-\chi_5(4a)=4\,.\label{dimrho5}\end{equation}

Consider now that $\rho_5$ is not irreducible; then $\rho_5\cong \rho'\oplus\rho''$ where $\rho'\in\{
\rho^{(5)}_4,\rho^{(5)}_5\}$ and $\rho''\in\{\rho_2^{(5)},\rho_3^{(5)},\rho^{(5)}_4,\rho^{(5)}_5\}$
and dim$\,\rho_5-\chi_5(4a)=6$ or 8, contradicting \eqref{dimrho5}.

Thus $\rho_5$ must be irreducible, with $0\in\cT(\rho_5)$, of dimension $\ge 5$, and even. The only 
possibility is $\rho_5\cong\rho^{(5)}_8$.
\end{proof}

A matrix realisation of $\rho_8^{(5)}$ is given in Appendix \ref{appendix:matrices}.

\section{End game}
\label{sec:endgame}

We have obtained in Propositions \ref{prop:group-of-12} and \ref{prop:group-of-4} that the modular
data $\rho$ of the centre of the extended Haagerup satisfies $\rho\iso \rho^{(13)}_{14}\oplus
\rho^{(5)}_8\oplus 1\oplus 1\oplus 1$.
Explicit matrix realisations of $\rho_{14}^{(13)}$ and of $\rho_8^{(5)}$ are in
Appendix \ref{appendix:matrices}.
Define $S'$ to be the corresponding block diagonal
matrix and
$T'$ to be the corresponding diagonal matrix. The statement that $\rho \iso \rho^{(13)}_{14}\oplus
\rho^{(5)}_8\oplus 1\oplus 1\oplus 1$ is that there is an invertible
22-by-22 matrix $Q$ so that $QS=S'Q$ and $QT=T'Q$.

We have established that the simples $\{\beta_i\}$ have $T$-eigenvalues the four primitive 5-th roots of unity;
as there is nothing to distinguish the $\beta_i$ amongst themselves we may assume the eigenvalues appear in any
convenient order.
Similarly, we know that the simples $\{\gamma_i\}\cup\{\delta_i\}\cup\{\epsilon_i\}$ have $T$-eigenvalues which are all
the primitive 13-th roots of unity. The $T$-eigenvalues for $\gamma_i$ determine the $T$-eigenvalues for $\delta_i$ and
$\epsilon_i$ since
\begin{align*}
T_{\delta_i \delta_i} & = T_{\gamma_i^{\sigma_{16}} \gamma_i^{\sigma_{16}}} = (T_{\gamma_i \gamma_i}^{16^2}) \\
\intertext{and}
T_{\epsilon_i \epsilon_i} & = T_{\gamma_i^{\sigma_{16}^2} \gamma_i^{\sigma_{16}^2}} = (T_{\gamma_i \gamma_i}^{16^4}).
\end{align*}
However it remains to decide which four of the 13-th primitive roots appear as the $T$-eigenvalues for the $\gamma_i$.
We look at top left entry of the equation $STS=CT^*S^*T^*$. The right hand side is simply $\cD^{-1}$, while
the left hand side becomes $\sum_{x \in \Phi} \frac{\dim(x)^2}{\cD^2} T_{xx}$. We find that this is only true if the
$$
 \frac{1}{2\pi i}\log(T_{\gamma_i \gamma_i})  =  \left( \frac {9}{13}, \frac{6}{13}, \frac{4}{13},
\frac{7}
{13} \right)
$$
(up to the permutation, which is fixed as shown). This we may take
\begin{equation}\label{Tfinal}
\frac{1}{2\pi i}\log(T_{xx}) = \left(0,0,0,0,0,0,\frac 1 5,\frac 2 5, \frac 3 5, \frac 4 5, \frac {9}{13}, \frac{6}{13}, \frac{4}{13}, \frac{7}
{13},\frac{3}{13}, \frac{2}{13}, \frac{10}{13}, \frac{11}{13},\frac{1}{13}, \frac{5}{13}, \frac{12}{13}, \frac{8}
{13}\right)\,.\end{equation}

We know the 16 simples $\{\beta_i\}\cup\{\gamma_i\}\cup\{\delta_i\}\cup\{\epsilon_i\}$ are all unique, in the sense of Section \ref{sec:basic-lemmas}, and so most entries of $Q$ are determined from Lemma \ref{2.1}.
The equation $QT = T'Q$ tells us that $Q$ is the product of a permutation and a block diagonal matrix with all blocks
1-by-1 except for one, corresponding to 1-eigenvalues of $T$, which is 6-by-6. Much of that 6-by-6 block is irrelevant. 

We also have learned much about $S$, some of which is collected in hypotheses (a)-(g) in the
following Theorem (e.g. we know $S^2=I$, since all simples are self-dual, so (f) is its $(x,x)$-entry 
for $x \in \{\gamma_i, \delta_i, \epsilon_i\}$). 

\begin{thm}
\label{thm:S}
Suppose
\begin{enumerate}[(a)]
\item $S'$ and $T'$ are the explicit matrices for $\rho^{(13)}_{14}\oplus
\rho^{(5)}_8\oplus 1\oplus 1\oplus 1$ appearing in Appendix \ref{appendix:matrices},
\item $T$ is the 22-by-22 diagonal matrix with entries given by Equation \eqref{Tfinal},
\item $S$ is a 22-by-22 matrix whose first three rows and columns are given by Equation \eqref{eq:dimensions} and
Theorem \ref{compatible}(2),
\item we have the modular invariants appearing in Equation \eqref{twoModInv},
\item $S$ is symmetric,
\item $\sum_y S_{xy} S_{xy} = 1$ for $x \in \{\gamma_i, \delta_i, \epsilon_i\}$,  and 
\item $Q$ is invertible and $QS=S'Q$ and $QT=T'Q$.
\end{enumerate}
Then $S$ is given by

\begin{align*}
S & = \left(
\begin{array}{ccc|ccccccc|ccc}
\multicolumn{3}{c|}{
	\multirow{3}{*}{
		\scalebox{1.5}{$U$}
	}
} &
$1/5$ &
$1/5$ &
$1/5$ &
$1/5$ &
$1/5$ &
$1/5$ &
$1/5$ &
\multicolumn{3}{c}{
	\multirow{3}{*}{
		\scalebox{1.5}{$V$}
	}
} \\
 & & &
$1/5$ &
$1/5$ &
$1/5$ &
$1/5$ &
$1/5$ &
$1/5$ &
$1/5$ &
 & & \\
 & & &
$1/5$ &
$1/5$ &
$1/5$ &
$1/5$ &
$1/5$ &
$1/5$ &
$1/5$ &
 & & \\ \hline
$1/5$ &
$1/5$ &
$1/5$ &
$4/5$ &
$-1/5$ &
$-1/5$ &
$-1/5$ &
$-1/5$ &
$-1/5$ &
$-1/5$ &
\multicolumn{3}{|c}{
	\multirow{7}{*}{
		\scalebox{1.5}{$0$}
	}
} \\
$1/5$ &
$1/5$ &
$1/5$ &
$-1/5$ &
$4/5$ &
$-1/5$ &
$-1/5$ &
$-1/5$ &
$-1/5$ &
$-1/5$ &
 \\
$1/5$ &
$1/5$ &
$1/5$ &
$-1/5$ &
$-1/5$ &
$4/5$ &
$-1/5$ &
$-1/5$ &
$-1/5$ &
$-1/5$ &
 \\ \cline{7-10}
$1/5$ &
$1/5$ &
$1/5$ &
$-1/5$ &
$-1/5$ &
$-1/5$ &
\multicolumn{4}{|c|}{
	\multirow{4}{*}{
		\scalebox{1.5}{$W$}
	}
} & &
\\
$1/5$ &
$1/5$ &
$1/5$ &
$-1/5$ &
$-1/5$ &
$-1/5$ &
\multicolumn{4}{|c|}{}
\\
$1/5$ &
$1/5$ &
$1/5$ &
$-1/5$ &
$-1/5$ &
$-1/5$ &
\multicolumn{4}{|c|}{}
\\
$1/5$ &
$1/5$ &
$1/5$ &
$-1/5$ &
$-1/5$ &
$-1/5$ &
\multicolumn{4}{|c|}{}
\\
\hline
\multicolumn{3}{c|}{
	\multirow{3}{*}{
		\scalebox{1.5}{$V^t$}
	}
} & 
\multicolumn{7}{c|}{
	\multirow{3}{*}{
		\scalebox{1.5}{$0$}
	}
} &
$A$ &
$B$ &
$C$ \\
& & & & & & & & & & $B$ & $-C$ & $A$ \\
& & & & & & & & & & $C$ & $A$ & $-B$
\end{array}
\right)
\end{align*}
with $U, V, W, A, B$, and $C$ given below.
\end{thm}
\begin{proof}
This calculation appears in {\tt code/EndGame.nb}, bundled with the {\tt arXiv} sources of this article.
We write $S=(S_{xy})_{x,y\in \Phi}$, and $Q=(Q_{ix})_{1 \leq i \leq 22, x \in \Phi}$.
The following simple steps completely identify $S$.
\begin{enumerate}
\item Solve the linear equations in the $\{S_{xy}\}$ coming from the modular invariants and symmetry.
\item Solve the linear equations in the $\{Q_{ix}\}$ coming from $QT=T'Q$ (this just shows that $Q$ is the product of a
permutation and a block diagonal matrix, as mentioned above).
\item Look at entries of $QS - S'Q$ which do not involve any of the remaining unknown $S_{xy}$; these are linear
equations in the $\{Q_{ix}\}$, which we can solve.
\item Observe that $\det Q$ has a factor of $Q_{1,\omega_0}$, so this must not be zero. Find all the equations coming
from $QS - S'Q$ of the form $Q_{1,\omega_0} X = 0$, where $X$ is a linear combination of the $\{S_{xy}\}$, and set $X=0$
for each.
\item Now, the equations $\sum_y S_{xy} S_{xy} = 1$ for $x \in \{\gamma_i, \delta_i, \epsilon_i\}$ simplify to $6
S_{\alpha_1 x}^2 = 0$ for these same $x$, so all these entries of the $S$-matrix must be zero.
\item Observe that $\det Q$ has a factor of $Q_{15,\omega_0}$, so this must not be zero. Find all the equations coming
from $QS - S'Q$ of the form $Q_{15,\omega_0} X = 0$, where $X$ is a linear combination of the $\{S_{xy}\}$, and set
$X=0$
for each.
\item Finally, treat the equations $QS - S'Q$ as quadratics in $\{S_{xy}\}$ and $\{Q_{ix}\}$ jointly, and solve them;
there are only 5 solutions, of which 4 make $\det Q = 0$. The remaining solution is the one described in the statement
of the Theorem. \qedhere
\end{enumerate}
\end{proof}

In fact, the same argument works if we disregard the modular invariant $$(1\,1\,1\,1\,1\,1\,0\,\ldots\,0),$$ although then at
the final step the quadratics have 64 solutions, of which only one allows $\det Q \neq 0$. We make this observation because 
there is a candidate third fusion category $EH3$ in the Morita equivalence class of the even parts of extended Haagerup.
One
can determine the fusion rules of this category, if it exists. The argument of \cite{1404.3955} determines the
dimensions of the irreducibles in $Z(EH3)$ (exactly the same as the dimensions here), and that $Z(EH3)$ would have the
modular invariant $(1\,1\,1\,2\,1\,0\,\ldots\,0)$, but not necessarily $(1\,1\,1\,1\,1\,1\,0\,\ldots\,0)$. Thus the fact
that the argument here does not rely on this second modular invariant shows that the centre of any fusion category with
the fusion rules of $EH3$ would have the same $S$ and $T$ matrices as the centre of extended Haagerup. Of course, the
$S$ and $T$ matrices are not known to be complete invariants of the centre. If they were, however, this discussion would
allow
one to establish the existence of a third category, Morita equivalent to $EH1$ and $EH2$, merely by constructing any
fusion category with the appropriate fusion ring.

In the above theorem describing $S$ we have, with $c_k = \cos(2 \pi k / 65)$,
\begin{align*}
U & = \begin{pmatrix}
u_1 & u_2 & u_3 \\
u_2 & u_3 & u_1 \\
u_3 & u_1 & u_2
\end{pmatrix},
\end{align*}
the $u_i$ are the roots of 
$21125 \lambda^3 - 8450 \lambda^2 + 585 \lambda - 1$,
\begin{align*}
u_1 & = \frac{1}{65}\left(7 - 10c_{20} - 10c_{30}\right) \\
	& \simeq 0.00175363 \displaybreak[1]\\
u_2 & = \frac{1}{65}\left(7 + 10c_{8} + 10c_{18} - 10c_{21} + 10c_{25} - 10c_{31}\right)\\
	& \simeq 0.311623 \displaybreak[1]\\
u_3 & = \frac{1}{65}\left(12 - 10c_{8} - 10c_{18} + 10c_{20} + 10c_{21} - 10c_{25} + 10c_{30} + 10c_{31}\right) \\
	& \simeq 0.0866238,
\end{align*}
\begin{equation*}
V = \left(\begin{array}{cccccccccccc}
v_1 & v_1 & v_1 & v_1 & -v_2 & -v_2 & -v_2 & -v_2 & -v_3 & -v_3 & -v_3 & -v_3 \\
v_2 & v_2 & v_2 & v_2 & -v_3 & -v_3 & -v_3 & -v_3 & -v_1 & -v_1 & -v_1 & -v_1 \\
v_3 & v_3 & v_3 & v_3 & -v_1 & -v_1 & -v_1 & -v_1 & -v_2 & -v_2 & -v_2 & -v_2
\end{array}\right),
\end{equation*}
the $v_i$ are the roots of $169 \lambda^3 - 13 \lambda - 1$, 
\begin{align*}
v_1 & = \frac{2}{13}\left(c_{8} + c_{18} +c_{20} - c_{21} +c_{25} +c_{30} -c_{31}\right) \\
	& \simeq 0.30969 \displaybreak[1]\\
v_2 & = \frac{1}{13}\left(1 - 4c_{8} - 4c_{18} + 2c_{20} + 4c_{21} - 4c_{25} + 2c_{30} +  4c_{31}\right) \\
    & \simeq -0.224999 \displaybreak[1]\\
v_3 & = \frac{1}{13}\left(-1 + 2c_{8} + 2c_{18} - 4c_{20} - 2c_{21} + 2c_{25} - 4c_{30} - 2c_{31}\right) \\
	& \simeq -0.0848702,
\end{align*}
\begin{align*}
W & = \frac{1}{10}\begin{pmatrix}
3-\sqrt{5} & -2-2\sqrt{5} & -2+2\sqrt{5} & 3+\sqrt{5} \\
-2-2\sqrt{5} & 3+\sqrt{5} & 3-\sqrt{5} & -2+2\sqrt{5} \\
-2+2\sqrt{5} & 3-\sqrt{5} & 3+\sqrt{5} & -2-2\sqrt{5} \\
3+\sqrt{5} & -2+2\sqrt{5} & -2-2\sqrt{5} & 3-\sqrt{5}
\end{pmatrix}
\end{align*}
and $A, B$, and $C$ are band matrices, so $A_{ij} = a_{i+j-1 \pmod{4}}$, etc., and $\{a_1, a_3, b_1, b_3, d_1, d_3\}$ are the roots of $28561 \lambda ^6-28561 \lambda ^5+8788 \lambda ^4-507 \lambda ^3-169 \lambda ^2+26 \lambda -1$, while $\{a_2, a_4, b_2, b_4, d_2, d_4\}$ are the roots of $28561 \lambda ^6-6591 \lambda ^4-507 \lambda ^3+338 \lambda ^2+39 \lambda -1$:
\begin{align*}
a_1 & = \frac{1}{13}\left(2 c_{8}-2 c_{10}+2 c_{18}-2 c_{20}-2 c_{21}+4 c_{25}-4 c_{30}-2 c_{31}+1\right) \displaybreak[1]\\
	& \simeq 0.07470114748\\
a_2 & = \frac{1}{13}\left(-2 c_{8}+2 c_{10}-2 c_{18}+6 c_{20}+2 c_{21}-2 c_{25}+2 c_{31}+1\right) \displaybreak[1] \\
	& \simeq 0.2714005479\\
a_3 & = \frac{1}{13}\left(2 c_{8}+2 c_{10}+2 c_{18}-2 c_{20}-2 c_{21}-2 c_{31}+2\right) \displaybreak[1] \\
	& \simeq 0.3520512456\\
a_4 & = \frac{1}{13}\left(-2 c_{10}-2 c_{20}+4 c_{30}\right) \displaybreak[1] \\
	& \simeq -0.1865303711\\
b_1 & = \frac{1}{13}\left(2 c_{10}-2 c_{20}-2 c_{25}-2\right) \displaybreak[1] \\ 
	& \simeq -0.1595713243\\
b_2 & = \frac{1}{13}\left(4 c_{8}+4 c_{18}+2 c_{20}-4 c_{21}-2 c_{25}-4 c_{31}\right) \displaybreak[1] \\ 
	& \simeq 0.3315913069\\
b_3 & = \frac{1}{13}\left(-2 c_{10}-2 c_{20}+2 c_{25}-4 c_{30}-3\right) \displaybreak[1] \\
	& \simeq -0.4369214224\\
b_4 & = \frac{1}{13}\left(-2 c_{8}-2 c_{18}+2 c_{21}+4 c_{25}+2 c_{30}+2 c_{31}\right) \displaybreak[1] \\
	& \simeq -0.02172236355\\
d_1 & = \frac{1}{13}\left(2 c_{8}+2 c_{10}+2 c_{18}-2 c_{21}+2 c_{30}-2 c_{31}-2\right) \displaybreak[1] \\
	& \simeq 0.1502976190\\
d_2 & = \frac{1}{13}\left(-6 c_{8}+6 c_{10}-6 c_{18}+4 c_{20}+6 c_{21}-4 c_{25}+4 c_{30}+6 c_{31}+2\right) \displaybreak[1] \\
	& \simeq 0.2171593392\\
d_3 & = \frac{1}{13}\left(2 c_{8}-2 c_{10}+2 c_{18}-2 c_{21}+4 c_{25}-2 c_{30}-2 c_{31}-3\right) \displaybreak[1] \\
	& \simeq -0.12705247914\\
d_4 & = \frac{1}{13}\left(2 c_{8}-6 c_{10}+2 c_{18}-2 c_{20}-2 c_{21}-2 c_{30}-2 c_{31}-1\right) \\
	& \simeq -0.4421581056.
\end{align*}

Finally, the matrix $Q$  is not uniquely determined; a nice choice is
\begin{align*} 
Q & = \left(\begin{smallmatrix}
 1 & 0 & -1 & 0 & 0 & 0 & 0 & 0 & 0 & 0 & 0 & 0 & 0 & 0 & 0 & 0 & 0 & 0 & 0 & 0 & 0 & 0 \\
 0 & 0 & 0 & 0 & 0 & 0 & 0 & 0 & 0 & 0 & 0 & 0 & 0 & 0 & 0 & 0 & 0 & 0 & -1 & 0 & 0 & 0 \\
 0 & 0 & 0 & 0 & 0 & 0 & 0 & 0 & 0 & 0 & 0 & 0 & 0 & 0 & 0 & -1 & 0 & 0 & 0 & 0 & 0 & 0 \\
 0 & 0 & 0 & 0 & 0 & 0 & 0 & 0 & 0 & 0 & 0 & 0 & 0 & 0 & 1 & 0 & 0 & 0 & 0 & 0 & 0 & 0 \\
 0 & 0 & 0 & 0 & 0 & 0 & 0 & 0 & 0 & 0 & 0 & 0 & 1 & 0 & 0 & 0 & 0 & 0 & 0 & 0 & 0 & 0 \\
 0 & 0 & 0 & 0 & 0 & 0 & 0 & 0 & 0 & 0 & 0 & 0 & 0 & 0 & 0 & 0 & 0 & 0 & 0 & -1 & 0 & 0 \\
 0 & 0 & 0 & 0 & 0 & 0 & 0 & 0 & 0 & 0 & 0 & -1 & 0 & 0 & 0 & 0 & 0 & 0 & 0 & 0 & 0 & 0 \\
 0 & 0 & 0 & 0 & 0 & 0 & 0 & 0 & 0 & 0 & 0 & 0 & 0 & 1 & 0 & 0 & 0 & 0 & 0 & 0 & 0 & 0 \\
 0 & 0 & 0 & 0 & 0 & 0 & 0 & 0 & 0 & 0 & 0 & 0 & 0 & 0 & 0 & 0 & 0 & 0 & 0 & 0 & 0 & -1 \\
 0 & 0 & 0 & 0 & 0 & 0 & 0 & 0 & 0 & 0 & -1 & 0 & 0 & 0 & 0 & 0 & 0 & 0 & 0 & 0 & 0 & 0 \\
 0 & 0 & 0 & 0 & 0 & 0 & 0 & 0 & 0 & 0 & 0 & 0 & 0 & 0 & 0 & 0 & 1 & 0 & 0 & 0 & 0 & 0 \\
 0 & 0 & 0 & 0 & 0 & 0 & 0 & 0 & 0 & 0 & 0 & 0 & 0 & 0 & 0 & 0 & 0 & -1 & 0 & 0 & 0 & 0 \\
 0 & 0 & 0 & 0 & 0 & 0 & 0 & 0 & 0 & 0 & 0 & 0 & 0 & 0 & 0 & 0 & 0 & 0 & 0 & 0 & -1 & 0 \\
 0 & 1 & -1 & 0 & 0 & 0 & 0 & 0 & 0 & 0 & 0 & 0 & 0 & 0 & 0 & 0 & 0 & 0 & 0 & 0 & 0 & 0 \\
 1 & 1 & 1 & -1 & -1 & -1 & 0 & 0 & 0 & 0 & 0 & 0 & 0 & 0 & 0 & 0 & 0 & 0 & 0 & 0 & 0 & 0 \\
 0 & 0 & 0 & 0 & 0 & 0 & -1 & 0 & 0 & 0 & 0 & 0 & 0 & 0 & 0 & 0 & 0 & 0 & 0 & 0 & 0 & 0 \\
 0 & 0 & 0 & 0 & 0 & 0 & 0 & -1 & 0 & 0 & 0 & 0 & 0 & 0 & 0 & 0 & 0 & 0 & 0 & 0 & 0 & 0 \\
 0 & 0 & 0 & 0 & 0 & 0 & 0 & 0 & -1 & 0 & 0 & 0 & 0 & 0 & 0 & 0 & 0 & 0 & 0 & 0 & 0 & 0 \\
 0 & 0 & 0 & 0 & 0 & 0 & 0 & 0 & 0 & -1 & 0 & 0 & 0 & 0 & 0 & 0 & 0 & 0 & 0 & 0 & 0 & 0 \\
 0 & 0 & 0 & 0 & 1 & -1 & 0 & 0 & 0 & 0 & 0 & 0 & 0 & 0 & 0 & 0 & 0 & 0 & 0 & 0 & 0 & 0 \\
 0 & 0 & 0 & 1 & 0 & -1 & 0 & 0 & 0 & 0 & 0 & 0 & 0 & 0 & 0 & 0 & 0 & 0 & 0 & 0 & 0 & 0 \\
 1 & 1 & 1 & 0 & 0 & 3 & 0 & 0 & 0 & 0 & 0 & 0 & 0 & 0 & 0 & 0 & 0 & 0 & 0 & 0 & 0 & 0
 \end{smallmatrix}\right)
\end{align*}
 
As a consistency check, we offer:
\begin{lem}
The Verlinde formula gives non-negative integer fusion multiplicities, which are consistent with the restriction functor
$Z(EH) \to EH$.
\end{lem}

\section{Character vectors}

A natural question is whether there is a vertex operator algebra (see e.g. \cite{MR2023933}) corresponding to the centre of the even part of
the extended Haagerup. This is at present too difficult to answer. However, in this section we obtain all possible
\textit{character vectors} with central charge $c\le 24$ compatible with the modular data computed 
in this paper. This should be information crucial for constructing the hypothetical vertex operator algebra,
or showing it cannot exist. Because the procedure for doing this is difficult to extract from the literature, we will include here
a more pedagogical treatment.

\subsection{The general theory}

By definition, a vertex operator algebra and its modules carry actions of the Virasoro algebra, so the vertex operator
algebra characters are expressible as combinations of Virasoro ones.
The Virasoro characters relevant to our discussion are given next.
When $c>1$ and $h>0$, there is a Virasoro irrep $V(c,h)$ with character
$$\operatorname{ch}_{V(c,h)}=q^{h-c/24}\prod_{n=1}^\infty (1-q^n)^{-1}=\sum_{m=0}^\infty p(m)q^{m+h-c/24}\,.$$
When $c>1$ and $h=0$, the Virasoro irrep $V(c,h)$ has character
$$\operatorname{ch}_{V(c,0)}(\tau)=q^{-c/24}\prod_{n=1}^\infty (1-q)(1-q^n)^{-1}=\sum_{m=0}^\infty (p(m)-p(m-1))q^{m-c/24}\,,$$ 
where $p(m)$ is the $m$th partition number, and where $q=e^{2\pi i\tau}$.

\begin{Def} Suppose $\rho$ is a $d$-dimensional representation of SL$(2,\bbZ)$ with $T=\rho \left(\begin{smallmatrix}1 & 1 \\ 0 & 1\end{smallmatrix}\right)$ a diagonal matrix. By a character vector $\bbX(\tau)$ for $\rho$, we mean: 
\begin{enumerate}[(i)]
\item $\bbX:\bbH\rightarrow\bbC^d$ is holomorphic throughout the upper half-plane 
$\bbH=\{\tau\in\bbC\,|\,\mathrm{Im}\,\tau>0\}$; 
\item there is a diagonal rational matrix $\lambda$ such that
\begin{equation}\label{qexp}
e^{-2\pi i\tau\lambda}\bbX(\tau)=\sum_{n=0}^\infty \bbX_ne^{2\pi in\tau}\end{equation}
converges absolutely in $\bbH$, and $\sum_{n=0}^\infty\bbX_nq^n$ is holomorphic at $q=0$; 
\item for all $\left(\begin{smallmatrix}a & b \\ c & d\end{smallmatrix}\right)\in\mathrm{SL}(2,\bbZ)$ and all
$\tau\in\bbH$, 
\begin{equation}\label{eq:functionalequ}
\bbX\left(\frac{a\tau+b}{c\tau+d}\right)=\rho\left(\begin{smallmatrix}a & b \\ c & d\end{smallmatrix}\right)\,
\bbX(\tau)\,;\end{equation}
\item each coefficient $\bbX_n$ takes values in $\bbZ_{\ge 0}^d$, $\lambda_{11}<\lambda_{jj}$ for all $j\ne 1$, and $(\bbX_0)_{1}=1$. Moreover, each component $\bbX_{(j)}(\tau)$ is nonzero and
can be written $\bbX_{(j)}(\tau)=\sum_{n=0}^\infty\bbX'_{n;j}\,\operatorname{ch}_{V(-24\lambda_{11},\lambda_{jj}-\lambda_{11}+n)}(\tau)$ 
where each $\bbX'_{n;j}\in\bbZ_ {\ge0}$.\end{enumerate}\end{Def}

It is common to write  $q^\lambda$ for $e^{2\pi i\tau\lambda}$. 
Note that $e^{2\pi i\lambda}=T$. A  function $\bbX(\tau)$ satisfying (i)-(iii) is called a \textit{weakly holomorphic
vector-valued modular function} for $\rho$ (`weakly holomorphic' means holomorphic in $\bbH$ and meromorphic at all cusps $\bbQ\cup\{i\infty\}$).
A consequence of the fact that the coefficients $\bbX_n$ are  rational, is  that $T$ has finite order (hence that $\lambda$ is rational).
The condition $\bbX'_n\in\bbZ_{\ge 0}^d$ implies $\bbX_n\in\bbZ_{\ge 0}^d$, but in practice 
isn't usually much stronger. We impose the condition $\lambda_{11}<\lambda_{jj}$ here because
we seek a unitary vertex operator algebra; if $\rho$ is modular data with $o\ne 1$ (recall Definition 2.1)
then this condition would become $\lambda_{oo}<\lambda_{jj}$.

Most representations $\rho$ will possess no character vectors; for example it is elementary to verify that it requires the first column of $S$ to be strictly positive, and an old conjecture of Atkin--Swinnerton-Dyer \cite{MR0337781} implies that the existence of a character vector is only possible
when ker$\,\rho$ contains some $\Gamma(N)$.

The modules of a (unitary) strongly-rational vertex operator algebra $\cV$ form a (unitary) modular tensor category \cite{MR2140309},
where `strongly-rational' means regular, simple, equivalent as a $\cV$-module to its contragredient $\cV^\vee$, $\cV_0=\bbC 1$ and $\cV_n=0$ for $n<0$.  The modules are infinite-dimensional,
but
 the operator $L_0$ in $\cV$ acts semi-simply on the modules, and
 the eigenspaces are all finite-dimensional. For each irreducible module $M$ of $\cV$, define the character
 $\chi_M(\tau)=q^{-c/24}\mathrm{tr}_Mq^{L_0}=q^{h_M-c/24}\sum_{n=0}^\infty\mathrm{dim}\,M_{h_M+n}\,q^n$,
 where $M=\coprod_{n=0}^\infty M_{h_M+n}$ and  $M_{h'}$ is the $L_0$-eigenspace with eigenvalue $h'$.
 The numbers $c,h_M$ are called the \textit{central charge} of $\cV$ and the \textit{conformal weight}
 of $M$. Then Zhu \cite{MR1317233} proved that these $\chi_M$ together  form a weakly holomorphic vector-valued
 modular function for some representation $\rho$ of SL$(2,\bbZ)$; this
 representation is given by the modular data of the modular tensor category \cite{1201.6644} (up to a third root
 of unity to be discussed shortly).  One irreducible $\cV$-module will be $\cV$ itself, which we make the
 first module.  The characters of the irreducible modules of a unitary strongly-rational vertex operator algebra, will form a character vector (hence the name).

The modular data  of a modular tensor category determines $T$ up to a third root
of unity.  This ambiguity  means that the
central charge is only determined up to a multiple of 8. In particular, if some vertex operator algebra
realises a modular tensor category, so will infinitely many others; once we've found a character
vector, we've found infinitely many others. For example, tensor arbitrary many copies
of the $E_8$ lattice vertex operator algebra to $\cV$; this doesn't change the category, but
 each copy increases the central charge by 8 and multiplies the character vector by $J(\tau)^{1/3}$.
  
Thus the first step to trying to recover a strongly-rational VOA $\cV=\coprod_{n=0}^\infty\cV_n$ from 
a modular tensor category is to select a possible $c$, and then determine the possible character vectors
$\chi_M(\tau)$. The second step would be to identify the space $\cV_1$. It will be a reductive Lie
algebra, and all homogeneous spaces $M_h$ of all $\cV$-modules
$M$ will be $\cV_1$-modules. The key formula for this purpose is Proposition 4.3.5 of  \cite{MR1317233}, which says 
that for all $u,v\in \cV_1$ and all $\cV$-modules $M$, 
\begin{equation}\sum_{n=0}^\infty \kappa_{M_{h_M+n}}(u,v)\,q^{n+h_M-c/24}=\mathrm{tr}|_{M}o(u[-1]v)\,q^{L_0-c/24}
+\frac{\langle u,v\rangle}{24}E_2(\tau)\,\chi_M(\tau)\,.\end{equation}
Here and elsewhere, $E_n(\tau)$ denotes the weight $n$  Eisenstein series for SL$(2,\bbZ)$,
normalised to have leading term 1.  Also, $\kappa_{M_h}(u,v)=\mathrm{tr}|_{M_h}o(u)u(v)$ is the Killing form of the $\cV_1$-module $M_h$; in particular,
$\kappa_{\cV_1}$ is the Killing form of $\cV_1$ itself. Closely related to the Killing form is the bilinear form $\langle u,v\rangle$, which is always nondegenerate and invariant. The first term on the right side is a vector-valued modular form
of weight 2, for the same multiplier $\rho$. By itself, $\cV_1$ generates a vertex operator subalgebra
of $\cV$, of affine algebra type.
 The {coset} or commutant of $\cV$ by this subalgebra should itself be a strongly-rational vertex operator algebra with small central
 charge and trivial Lie algebra part and explicitly known character vector. Constructing $\cV$
  then largely comes down  to identifying that coset vertex operator algebra. 
  
For both steps 1 and 2, constructing vector-valued modular forms is crucial. In this paper we will restrict our attention to determining the possible character vectors, although the same
 method determines the possible weight-2   forms.
The following treatment is  developed in \cite{MR2412268,vvmf}.

Fix an SL$(2,\bbZ)$-representation $\rho$ with $T$ diagonal and of finite order. Let $\cM^!(\rho)$
denote the space of all weakly holomorphic vector-valued modular functions for $\rho$. Let
$J(\tau)=q^{-1}+744+196884q+\cdots$ denote the Hauptmodul for SL$(2,\bbZ)$. In particular, $\cM^!(1)=
\bbC[J(\tau)]$. Note that   $\cM^!(\rho)$ is a module for the ring $\bbC[J(\tau)]$. Note that if $\rho'=
Q\rho Q^{-1}$, then $\bbX(\tau)\in\cM^!(\rho)$ iff $Q\bbX(\tau)\in\cM^!(\rho')$.

A simple observation: if $\rho$ is an \textit{odd} SL$(2,\bbZ)$-irrep, then  $\cM^!(\rho)=0$. This is because
\eqref{eq:functionalequ} applied to $\begin{psmallmatrix}a & b \\ c & d\end{psmallmatrix}=\begin{psmallmatrix}-1 & 0 \\ 0 & -1\end{psmallmatrix}$ gives
$\bbX(\tau)=-\bbX(\tau)$. For this reason, in the following we'll restrict (without loss of generality) to even representations $\rho$ by first projecting
away any odd summands. Conveniently, the modular data we obtain from the extended Haagerup subfactor is already even.

The first fact is that $\cM^!(\rho)$ is a  free module of rank $d$ over $\bbC[J(\tau)]$. Given $d$ generators, it is
convenient to collect them together as columns of a $d$-by-$ d$ matrix we'll call $\Xi(\tau)$; then there is
a bijection between $\bbX(\tau)\in\cM^!(\rho)$ and vectors $\bbY(\tau)\in\bbC^d[J(\tau)]$ given by $\bbX=\Xi\bbY$. 
We can choose the  generators (hence $\Xi$) in such a way that there is a diagonal matrix $\Lambda$
such that
\begin{equation}
\label{eq:exponent}
\Xi(\tau)=q^\Lambda\left(I+\sum_{n=1}^\infty \Xi_nq^n\right).
\end{equation}
Identifying any such $\Xi$  is equivalent to identifying the full space $\cM^!(\rho)$.

A word of warning: the convention of \eqref{eq:exponent} differs from that of \cite{MR2412268,MR2837122} which used $\Xi=q^\Lambda(Iq^{-1}+\chi+...)$, but is the same as in \cite{vvmf}. This notation change cleans up the formulas a little.
It is not completely trivial that generators can be chosen so that \eqref{eq:exponent} holds, but indeed it is true for all
$\bbC[J]$-submodules $M$ of $\cM^!(\rho)$ of full rank.  Once one has $d$ vector-valued modular forms $\bbX^{(i)}(\tau)$
 in $M$ 
forming a matrix $\Xi$ of shape \eqref{eq:exponent} for some $\Lambda$, it is then elementary to find algorithmically
$d$ free generators for $M$ with shape \eqref{eq:exponent} (for a larger $\Lambda$) as desired. The (nonconstructive) existence of such  $\bbX^{(i)}(\tau)$ is an immediate consequence of Theorem 3.1 of \cite{vvmf}. 

The second fact is that $\Xi(\tau)$ is the solution to a first-order Fuchsian differential equation. The reason
is that $\frac{E_{10}(\tau)\,d}{\Delta(\tau)\,d\tau}$ is a differential operator on $\cM^!(\rho)$, and so applied to each of the
free generators (i.e. columns of $\Xi$) gives a vector-valued modular form which lies in the $\bbC[J(\tau)]$-span of the
generators.
That differential equation implies the
recursion
\begin{equation}\label{eq:recurs}
[\Lambda,{\Xi}_{n}]+n{\Xi}_{n}=\sum_{l=0}^{n-1}{\Xi}_{l}
\left(f_{n-l}\Lambda+g_{n-l}(\Xi_1+[\Lambda,\Xi_1])\right)
\end{equation}
for $n\ge 2$, where 
we write  $(J(\tau)-984)\Delta(\tau)/E_{10}(\tau)=\sum_{n=0}^\infty f_{n}q^n=1+0q+338328 q^2 +\cdots$ and
$\Delta(\tau)/E_{10}(\tau)=\sum_{n=0}^\infty g_{n}q^n=q+240 q^2 + 199044 q^3 + \cdots$. Here,
$\Delta=\eta^{24}$ where $\eta$ is the Dedekind eta.
We require ${\Xi}_{0}=I$. Note that the $ij$-entry on the left-side of \eqref{eq:recurs} is
$\left(\Lambda_{ii}-\Lambda_{jj}+n\right){\Xi}_{n\, ij}$, so  \eqref{eq:recurs} allows us to recursively
identify all  entries of ${\Xi}_{n}$, at least when all 
$|\Lambda_{jj}-\Lambda_{ii}|\ne n$.   Indeed, it can be shown that
 $\Lambda_{jj}-\Lambda_{ii}$ can never lie in
$\bbZ_{\ge 2}$.

This recursion means  that $\cM^!(\rho)$ is completely identified, i.e. $\Xi(\tau)$ is determined,
once  the matrices $\Lambda$ and $\Xi_1$ are given. The matrices $\Xi_1$ and $\Lambda$ are heavily constrained. In particular, $\Lambda$ is diagonal, satisfying $e^{2\pi i \Lambda}=T$ as well as 
\begin{equation}\label{trace}\mathrm{Tr}\,\Lambda=-\frac{7d}{12}+\frac{1}{4}\mathrm{Tr}\,{S}+
\frac{2}{3\sqrt{3}}\mathrm{Re}\left(e^{\frac{-\pi i}{6}}
\mathrm{Tr}\,{ST^{-1}}\right)\,.\end{equation}
When $\rho$ is  irreducible and $d<6$, then any diagonal matrix satisfying both $e^{2\pi i\Lambda}=T$ and 
\eqref{trace} will work, but in general these conditions won't always suffice.

Any $\Xi(\tau) \in \cM^!(\rho)$ whose components are linearly independent over $\bbC$ gives us all of $\cM^!(\rho)$ via \cite[Proposition 3.2]{vvmf}: $$\cM^!(\rho) = \bbC[J(\tau), \nabla_{1}, \nabla_{2}, \nabla_{3}],$$
where
\begin{align*}
\nabla_{1} & = \frac{E_4 E_6}{\Delta} q \frac{d}{dq} \\ 
\nabla_{2} & = \frac{E_4^2}{\Delta} \left(q \frac{d}{dq}- \frac{E_2}{6} \right) q \frac{d}{dq}\\ 
\nabla_{3} & = \frac{E_6}{\Delta} \left(q \frac{d}{dq} - \frac{E_2}{3} \right) \left(q \frac{d}{dq} - \frac{E_2}{6} \right) q \frac{d}{dq} \\ 
\end{align*} 
The building blocks of all of these differential operators is the operator $q \frac{d}{dq}-\frac{k}{12}E_2$,
which sends weight $k$ modular forms to weight $k+2$ ones.

We may take $\Lambda_{\rho_1 \oplus \rho_2} = \Lambda_{\rho_1} \oplus \Lambda_{\rho_2}$ and $(\Xi_{\rho_1 \oplus \rho_2})_{1} = (\Xi_{\rho_1})_{1} \oplus (\Xi_{\rho_2})_{1}$. Moreover, $\Lambda$ and $\Xi$ for the weakly-holomorphic vector-valued modular forms at weight 2 for the contragredient representation $\rho$, is $-I-\Lambda$ and
$E_4(\tau)^2E_6(\tau)\Delta(\tau)^{-1}(\Xi(\tau)^\transpose)^{-1}$. (Definition 10.1 can be extended to forms of arbitrary even weight in 
the obvious way; Proposition 4.1 of \cite{vvmf} tells how to convert $\Xi(\tau)$'s for different weights but the same $\rho$.)

We can find $\Xi_{\rho_1 \otimes \rho_2}(\tau)$ from $\Xi_{\rho_1}(\tau)$ and $\Xi_{\rho_1}(\tau)$ using the fact  that
$$\cM^!(\rho_1 \otimes \rho_2) = \bbC[J(\tau), \nabla_{1}, \nabla_{2}, \nabla_{3}] (\cM^!(\rho_1) \otimes \cM^!(\rho_2)).$$
Suppose $\cM^!(\rho_1)$ is free of rank $d_1$ over $\bbC[J(\tau)]$, and $\cM^!(\rho_2)$ is free of rank $d_2$ over $\bbC[J(\tau)]$.
Starting with the matrix
$\tilde{\Xi} = \Xi_{\rho_1}(\tau) \otimes \Xi_{\rho_2}(\tau) \in M_{d_1d_2 \times d_1d_2}(\bbC((q)))$,
we form the $d_1d_2 \times 4d_1d_2$ matrix 
$\begin{psmallmatrix}\tilde{\Xi} & \nabla_{1} \tilde{\Xi} & \nabla_{2} \tilde{\Xi} & \nabla_{3} \tilde{\Xi}\end{psmallmatrix}$
and then find a $\bbC[J(\tau)]$ basis for the columns.
%
Replacing $\tilde{\Xi}$ with these new basis vectors as columns, we repeat until $\tilde{\Xi}$ stabilises.
This does not quite provide our $\Xi_{\rho_1 \otimes \rho_2}(\tau)$, as we still need to perform a change of basis so that Equation \eqref{eq:exponent} holds.


In the case of any irrep $\rho$ with kernel containing  $\Gamma(N)$ for some $N=\prod_pp^{\nu_p}$, we write $\rho\cong \oplus_i\rho_i$
and $\rho_i\cong \otimes_p\rho_{i;p}$ as before. It then suffices to know 
the $\Lambda$ and $\chi$ for each irrep $\rho_{i;p}$ appearing in that decomposition. Each such
$\rho_{i;p}$ is an irrep in some Weil representation associated to lattices, and so some $\bbX(\tau)\in
\cM^!(\rho_{i;p})$ with linearly independent components can be built up from lattice theta functions.
For `small' powers $p^\nu$, $\Lambda,\Xi_1$ have been computed for every irrep $\rho$ of SL$(2,\bbZ_{p^\nu})$, by Timothy Graves in his PhD thesis.
This means that the full space $\cM^!(\rho)$ can be determined fairly quickly from his tables for any representation of
SL$(2,\bbZ_N)$, provided the prime powers dividing $N$ are not too large ($<32$).

A minor technicality: it is possible for the tensor product $ \otimes_p\rho_{i;p}$ to be even, even though some (necessarily 
an even number of) $\rho_{i;p}$ 
may be odd. One way to handle this is to replace any such odd factor $\rho_{i;p}$ with the even irrep $\rho_2^{(2)}\otimes
\rho_{i;p}$, as an even number of $\rho_2^{(2)}$'s tensor to 1. For the modular data associated to the extended Haagerup subfactor, all components $\rho_{i;p}$ which arise are even.

These calculations can be a little delicate. We suggest two strong consistency checks. First,
\begin{eqnarray}
\cA_2\left(\cA_2-\frac{1}{2}I\right)&=0\,,\label{A2rel}\\
\cA_3\left(\cA_3-\frac{1}{3}I\right)\left(\cA_3-\frac{2}{3}I\right)&=0\,,\label{A3rel}
\end{eqnarray}
where
\begin{equation}
\cA_2=-\frac{31}{72}\Lambda-\frac{1}{1728}(\Xi_1+[\Lambda,\Xi_1])\,,\ \
\cA_3=-\frac{41}{72}\Lambda+\frac{1}{1728}(\Xi_1+[\Lambda,\Xi_1])\,.\label{A2A3}
\end{equation}
Given $\Lambda$ and $\Xi_1$, construct $\Xi(\tau)$ through \eqref{eq:recurs}; then the columns of
$\Xi(\tau)$  will freely generate the $\bbC[J]$-module $\cM^!(\rho)$ for some SL$(2,\bbZ)$-representation $\rho$,
iff the corresponding $\cA_2,\cA_3$ satisfy \eqref{A2rel},\eqref{A3rel}. Incidentally, $e^{2\pi i \cA_2}$ is similar to
$S$ and $e^{2\pi i \cA_3}$ is similar to $\rho\begin{psmallmatrix}-1&1\\ -1&0\end{psmallmatrix}=TS$.
This representation $\rho$ is, as always, uniquely
determined by its values on $\begin{psmallmatrix}1 & 1 \\ 0 & 1 \end{psmallmatrix}$ (which is $T=e^{2\pi i\Lambda}$) and $
\begin{psmallmatrix}0 & -1 \\ 1 & 0 \end{psmallmatrix}$ (which is $S$). The $S$-matrix can be estimated numerically by using the recursion \eqref{eq:recurs}
to compute the first few terms of the series expansion of $\Xi(\tau)$; then $\Xi(\tau)$ is invertible anywhere in $\bbH$
except at the countably many elliptic fixed points SL$(2,\bbZ).i\cup\mathrm{SL}(2,\bbZ).\xi_3$, so as
long as we avoid those elliptic points we can estimate $S=\Xi(-1/\tau)\,\Xi(\tau)^{-1}$. We have applied both tests
to all $\Lambda,\Xi_1$ given below.

Since the central charge $c$ is determined only up to mod 8 by the modular tensor category, there always
are three SL$(2,\bbZ)$-representations which have to be considered, namely $\rho,\rho_2^{(3)}\otimes\rho$, and
$\rho_3^{(3)}\otimes\rho$, where $\rho^{(3)}_2,\rho^{(3)}_3$ are described in Section \ref{sec:congruence-reps}. There is no straightforward
relation between the matrices $\Xi(\tau)$ for these three representations. However, Proposition 4.1(2) of
\cite{vvmf} gives a short-cut. Suppose we know $\Xi(\tau)$ for $\rho$; then the columns for $\Xi(\tau)$ for $\rho_2^{
(3)}\otimes\rho$ will be linear combinations over $\bbC$ of the columns of $E_4(\tau)\eta(\tau)^{-8}\Xi(\tau)
=q^\Lambda(I_d+\cdots)$ and of
$$\eta(\tau)^{-8}\left(E_4(\tau)\Xi(\tau)\Lambda\left(\Lambda-\frac{1}{6}I_d\right)-\left(q\frac{d}{dq}-\frac{1}
{6}E_2 (\tau)\right)q\frac {d} {dq}\Xi(\tau)\right)=q^\Lambda\left(1728\cA_3\left(\cA_3-\frac{1}
{3}I_d\right)q+\cdots\right)\,.$$
Now $\operatorname{rank}(\cA_3(\cA_3-\frac{1}{3}I_d))$ equals the multiplicity of $\xi_{3}^2$ as an eigenvalue of $TS$,
and
this is
the number of vectors that should be chosen from the latter. This method applied to $\rho_2^{(3)}\otimes\rho$ gives $\Xi(\tau)$ for $\rho_3^{(3)}\otimes\rho$.

Obtaining the possible character vectors is now easy combinatorics. Suppose $\bbX(\tau)$ is a character vector,
and write $\bbX(\tau)=\Xi(\tau)\,\bbY(J(\tau))$ for some vector-valued polynomial $\bbY(J)\in\bbC^d[J]$. Write
$d_j$ for the degree of the component $\bbY_j$, and $d_M$ for the maximum of all $d_j$. Then $d_M\le
c/24+\mathrm{max}_j\Lambda_{jj}$. More precisely, if $d_1=d_M$, then $d_1=c/24+\Lambda_{11}$ and $\bbY_1$ is monic. If $d_j=d_M$ and $j\ne 1$, then $d_j<c/24 +\Lambda_{jj}$ and the leading coefficient of $\bbY_j$
must be a positive integer. 

Given some $\bbX(\tau)\in\cM^!(\rho)$, write $\bbX(\tau)=q^\lambda\sum_{n=0}^\infty \bbX_nq^n$ where
each entry of $\bbX_0$ is nonzero. 
To prove a candidate $\bbX(\tau)$ is indeed a character vector, we need to prove each $\bbX_n\in\bbZ^d$,
and that each $\bbX_n\in \bbR^d_{\ge 0}$. The first statement is accomplished by:

\begin{lem}
\label{lem:integrality} 
Suppose $f(\tau)=q^\lambda\sum_{n=0}^\infty f_nq^n$ is a (scalar-valued) weakly holomorphic modular function 
for some subgroup $\Gamma$ of SL$(2,\bbZ)$, possibly with multiplier $\mu:
\Gamma\rightarrow\bbC^\times$. 
Suppose $\Gamma$ has index $m$ in SL$(2,\bbZ)$, and contains a congruence subgroup.
Choose $k\in\bbZ_{\ge0}$ so that $\lambda\ge -k/24$ for all $j$. Suppose the Fourier coefficients 
$f_n$ are integral for all $n\le km/24$. Then $f_n\in\bbZ$ for all $n$.\end{lem}

This is Lemma 3(b) of \cite{1211.5531}, applied to $\eta(\tau)^kf(\tau)$. A useful fact is that the index of
$\Gamma(N)$ in SL$(2,\bbZ)$ is $N^3\prod_{p|N}
(1-p^{-2})$. We apply the Lemma by taking $f(\tau)$ to be any component $\bbX_j(\tau)$ of our vector-valued
modular function $\bbX$, so $\lambda=\lambda_j$ and $\Gamma$ is the projective kernel of
the multiplier $\rho$.

Positivity is more delicate, and again follows the methods of \cite{1211.5531}. The general argument will be developed
elsewhere, and here we will limit the discussion to the following. Assume $\rho$ is a unitary SL$(2,\bbZ)$-representation, and that $S_{1j}>0$ for all $j$. (This is true for the modular data of any unitary modular tensor category.) Assume also that $\lambda_{11}<\lambda_{jj}$
for all $j\ne 1$ (this is true for any character vector). Then for large $n$, the Rademacher expansion for $\bbX(\tau)$
implies
$$(\bbX_n)_j\sim S_{1j}(\bbX_0)_1\frac{e^{4\pi\sqrt{n|\lambda_{11}|}}}{\sqrt{2}n^{3/4}}\,.$$
Hence for all sufficiently large $n$, all coefficients $\bbX_n$ will be positive, provided $\bbX_0\in\bbR^d_{>0}$.
To prove a given $\bbX$ is truly a character vector, we would need to make this estimate effective. This can
be quite involved, and will be treated in generality in future work following the positivity method developed
in \cite{1211.5531}. We will not address this further in this paper. 

\subsection{Specialisation to the double of the even part of EH}
\label{sec:EH-characters}

For the modular data of the extended Haagerup, the central charge $c$ will be a multiple of 8 (a
positive multiple, if we insist, as we will, that the hypothetical vertex operator algebra be unitary).
The corresponding conductor $N$ will be $N=5\cdot 13$ if $24|c$ or $N=3\cdot 5\cdot 13$ otherwise.
We have the decompositions:

$\rho\cong \rho_{14}^{(13)} \oplus \rho_8^{(5)}\oplus   \rho_1^{(1)}\oplus   \rho_1^{(1)}\oplus  \rho_1^{(1)}$ if
$c\equiv 0$ (mod 24);

$\rho\cong \left(\rho_2^{(3)}\otimes\rho_{14}^{(13)}\right) \oplus \left( \rho_2^{(3)}\otimes \rho_8^{(5)}\right)\oplus
\rho_2^{(3)}\oplus   \rho_2^{(3)}\oplus  \rho_2^{
(3)}$ if
$c\equiv 8$ (mod 24);

$\rho\cong  \left(\rho_3^{(3)}\otimes\rho_{14}^{(13)}\right) \oplus
\left(\rho_3^{(3)}\otimes\rho_8^{(5)}\right)\oplus \rho_3^{(3)}\oplus   \rho_3^{(3)}\oplus  \rho_3^{(3)} $ if
$c\equiv 16$ (mod 24).

Moreover, 
$$\Lambda(\rho_1^{(1)})=\Xi_1(\rho_1^{(1)})=(0)\,;$$

$$\Lambda(\rho_2^{(3)})=(-1/3)\,, \ \Xi_1(\rho_2^{(3)})=(248)\,;$$

$$\Lambda(\rho_3^{(3)})=(-2/3)\,, \ \Xi_1(\rho_3^{(3)})=(496)\,;$$

$$\Lambda(\rho_8^{(5)})=\mathrm{diag}\left(0,-\frac{4}{5},-\frac{3}{5},-\frac{2}{5},-\frac{6}{5}\right)\,,\ \Xi_1(\rho_8^{(5)})=\left(\begin{matrix} 25& -57750&-11550&-1350&-819000\cr-3/2&-39&-126&-7&468\cr-5/3&-1050&248&-9&1950\cr5&1650&264&282&-28600\cr -1/6&-7&-4&-3&-12\end{matrix}\right)$$

$$\rho_2^{(3)}\otimes\rho_8^{(5)}:\ \Lambda=\mathrm{diag}\left(-\frac{1}{3},-\frac{2}{15},-\frac{14}{15},-\frac{11}{15},-\frac{8}{15}\right)\,,\ \Xi_1=\left(\begin{matrix}-52&-30&-22050&-6600&-1680\cr-100&0&-39200&3850&1728\cr-5&-4&56&11&16\cr-10&2&84&220&-108\cr-25&8&1200&-1100&32\end{matrix}\right);$$

$$\rho_3^{(3)}\otimes\rho_8^{(5)}:\ \Lambda=\mathrm{diag}\left(-\frac{2}{3},-\frac{7}{15},-\frac{4}{15},-\frac{16}{15},\frac{2}{15}\right)\,,\ \Xi_1=\left(\begin{matrix}-29&-294&-60&-2640&3\cr -375&56&-50&3300&-1\cr-2025/2&-686&82&4312&-1/2\cr-25/2&14&2&-104&-1/2\cr-6125&2401&100&-411600&3
\end{matrix}\right)$$

$$\Lambda(\rho_{14}^{(13)})=\mathrm{diag}\left(0,-\frac{12}{13},-\frac{11}{13},-\frac{10}{13},-\frac{9}{13},-\frac{8}{13},-\frac{7}{13},-\frac{6}{13},-\frac{5}{13},-\frac{4}{13},-\frac{3}{13},-\frac{15}{13},-\frac{14}{13},0\right)\,,$$
$$ \Xi_1(\rho_{14}^{(13)})={\left(\begin{smallmatrix}   61/3&16731&-36374& 20748& 8281& 1703&-2145&962&169&-117&39&-276822&45474&-29/3\cr-2/3&-4&35&-12& 21& -30& 0& 20&-7& 0&-5&-33&44&4/3\cr 1&103&42&35&-14&45&-22&3&6&15&-10&-84&-66&-1\cr 5&-1&168&-96&-182&-27&0&-34&21&27&-8&-512&-330&-3\cr 13/3&203&70&-330&-18&161&78&14& -14& 0&-14& -924&-286&-8/3\cr -1&-805&585&-308&365&231&-55&-30&21&11&10&462&66&1\cr 5&1000&-1230&-693&168&-75&246&-60&48&54&-10&-1836&-3732&-4\cr -14/3&3276&721&0&637&-210&-99& 200&42&0&39&9702&-3432&4/3\cr -4&278&522&1034&-658&324&154&114&162&-88&-2&14124&-9186&1\cr 1&3093&4299&1260&-1729&633&468&150&-189&21&39&-1554&-5946&-1\cr -14/3&-11375&-8980&1617&-1078&770&0&440&-154&99&-4&25872&-5929&7/3\cr-2/3&-5&-10&6&8&-2&-6&5&0&-3&2&-12&-5&1/3\cr -1&-13&-9&0&14&-9&-13&6&-6&-5&1&-6&-18&1\cr 29/3&62985&-26962&15288&-3094&6175&858&-494&650&117&39&-204516&172458&-7/3 \end{smallmatrix}\right)}$$
  
$$\Lambda(\rho_2^{(3)}\otimes\rho_{14}^{(13)})=\mathrm{diag}\left(-\frac{1}{3}, -\frac{10}{39}, -\frac{7}{39}, -\frac{4}{39}, -\frac{40}{39}, -\frac{37}{39}, -\frac{34}{39}, -\frac{31}{39}, -\frac{28}{39}, -\frac{25}{39}, -\frac{22}{39}, -\frac{19}{39}, -\frac{16}{39}, -\frac{1}{3}\right)\,,$$
$$\Xi_1(\rho_2^{(3)}\otimes\rho_{14}^{(13)})={\left(\begin{smallmatrix}  -18&15&-24&5&11960&2854&-5083&3212&685&-1190&924&-536&62&78\cr -27&20&16&-2&6320&-10695&4420&5533&150&645&-880&-258&300&101\cr
-98&50&6&4&1810&17990&-13770&3010&842&2110&-1540&-284&-175&-34\cr
112&-16&15&-2&-56192&-21651&-15912&-649&3038&1473&616&-1082&-420&50\cr
4&0&1&-2&0&10&0&7&-2&-5&0&-8&-4&-1\cr
2&-5&4&-3&40&42&-17&-12&13&10&11&-4&-2&6\cr
-7&4&-4&-2&48&-35&68&-11&22&29&0&-13&-20&5\cr
14&10&2&-1&195&-90&-34&118&31&30&44&36&-15&-10\cr
-7&0&1&2&-320&189&136&77&194&-135&-56&54&-28&18\cr
-28&5&6&3&-1040&476&442&198&-337&40&154&4&-14&20\cr
21&-20&-11&2&-880&761&52&638&-342&375&-88&75&-12&11\cr
-38&-14&-6&-3&-3404&514&-782&1386&733&-18&176&64&80&-10\cr
-19&40&-8&-4&-2720&271&-3536&-1122&-1076&-255&-88&194&56&60\cr
56&60&-18&7&-4160&10688&1768&-1166&2675&440&693&-376&274&20 \end{smallmatrix}\right)}$$

$$\Lambda(\rho_3^{(3)}\otimes\rho_{14}^{(13)})=\mathrm{diag}\left(-\frac{2}{3}, -\frac{23}{39}, -\frac{20}{39}, -\frac{17}{39}, -\frac{14}{39}, -\frac{11}{39}, -\frac{8}{39}, -\frac{5}{39}, -\frac{2}{39}, -\frac{38}{39}, -\frac{35}{39}, -\frac{32}{39}, -\frac{29}{39}, -\frac{2}{3}\right)\,,$$
$$\Xi_1(\rho_3^{(3)}\otimes\rho_{14}^{(13)})={\left(\begin{smallmatrix}  -39&40&-128&72&35&5&-10&4&1&-760&749&-564&100&155\cr
-78&80&98&-14&14&-20&7&6&0&532&-924&-351&546&260\cr
-312&234&64&65&0&30&-16&3&0&2184&-2002&-480&-390&-104\cr
455&-92&160&-68&-98&-28&-16&0&1&1862&980&-2240&-1160&182\cr
975&235&16&-254&21&60&10&6&-1&-6840&-2058&-2700&-750&-715\cr
-403&-884&523&-208&169&76&-9&-4&1&6916&3796&962&156&1443\cr
-2119&832&-896&-351&91&-28&30&-2&1&13832&728&-3108&-5174&1547\cr
3042&2392&460&-26&182&-44&-8&10&1&15067&16380&12960&-3770&-2197\cr
-1170&169&280&351&-182&57&20&5&2&-55328&-20475&15560&-8385&4342\cr
-13&2&8&5&-7&3&2&1&-1&0&21&4&-9&13\cr
13&-24&-19&6&-7&4&1&2&-1&76&-28&30&-10&13\cr
-39&-26&-16&-26&-13&2&-2&3&1&0&78&28&65&-13\cr
-26&92&-28&-34&-7&0&-8&-2&-1&-133&-56&160&58&91\cr
118&179&-88&57&-14&25&4&-1&2&304&574&-408&389&38\end{smallmatrix}\right)}$$
  
In all three cases (namely, $c\equiv 8k$ (mod 24) for $k=0,1,2$), the full 22-by-22 matrices $\Lambda$ and $\Xi_1$ are obtained by
\begin{align*}
\Lambda& =Q^{-1}\left(\Lambda(\rho_{13})\oplus\Lambda(\rho_5)\oplus\Lambda(\rho_1)\oplus\Lambda(\rho_1)\oplus\Lambda(\rho_1)
\right)Q \\
\intertext{ and }
\Xi_1& =Q^{-1}\left(\Xi_1(\rho_{13})\oplus\Xi_1(\rho_5)\oplus\Xi_1(\rho_1)\oplus\Xi_1(\rho_1)\oplus\Xi_1(\rho_1)
\right)Q,
\end{align*}
for $Q$ explicitly given in Section \ref{sec:endgame}, and where $\rho_1=\rho_2^{(3)\,\otimes k}$, $\rho_5=\rho_2^{(3)\,\otimes k}\otimes \rho_8^{(5)}$, and $\rho_{13}=\rho_2^{(3)\,\otimes k}\otimes \rho_{14}^{(13)}$.

Let us explain how we found these matrices $\Lambda$ and $\Xi_1$. Consider first the $A_4$ root lattice
and its dual $A_4^*$ (we use the standard lattice notation and terminology explained in e.g. \cite{MR1662447}). The group $A_4^*/A_4$ has 5 elements, and these have theta series 
$\theta_{[0]}(\tau)=1+40q+\cdots$, $\theta_{[1]}(\tau)=\theta_{[4]}(\tau)=q^{2/5}(5+30q+\cdots)$
and $\theta_{[2]}(\tau)=\theta_{[3]}(\tau)=q^{3/5}(10+25q+\cdots)$, where $\theta_{[i]}=\theta_{[5-i]}$ follows
because a coset and its negative always have identical theta series. These 3 functions form
the components of a vector-valued modular form of weight 2 for SL$(2,\bbZ)$, for a multiplier equivalent to $\rho_5^{(5)}$. The products $\theta_{[i]}(\tau)\theta_{[j]}(\tau)$ will form a vector-valued modular
form of weight 4 for SL$(2,\bbZ)$, for a multiplier equivalent to the symmetric square of $\rho_5^{(5)}$.
That symmetric square is isomorphic to $1\oplus\rho_8^{(5)}$. We can make them weight 0 by dividing by
$\eta(\tau)^8$, but this tensors the multiplier by $\rho^{(3)}_2$. In particular, the first column of the
matrix $\Xi(\tau)$ for $\rho^{(3)}_2\otimes\rho_8^{(5)}$ has components $(\theta_{[0]}^2-2\theta_{[1]}
\theta_{[2]})\eta^{-8}$, $-\theta_{[2]}^2\eta^{-8}$, $-\theta_{[0]}\theta_{[1]}\eta^{-8}$, $-\theta_{[0]}\theta_{[2]}
\eta^{-8}$, $-\theta_{[1]}^2\eta^{-8}$. This generates the full module $\cM^!(\rho^{(3)}_2\otimes\rho_8^{(5)})$, using the differential operators $\nabla_i$ and $\bbC[J]$. 

A similar method works to find  $\Lambda,\Xi_1$ for $\rho^{(13)}_{14}$. For this let the lattice
be $L=A_{3}52[1,\frac{1}{4}]$, which means $\cup_{i=0}^3L_0+([i],\frac{i}{4})$ for the orthogonal
direct sum $L_0=A_3\oplus
\sqrt{52}\bbZ$. Then $L^*/L$ has 13 elements, with theta functions $\psi_{[0]}(\tau)=1q^0+\cdots$,
$\psi_{[1]}(\tau)=\psi_{[12]}(\tau)=1q^{2/13}+\cdots$, $\psi_{[2]}(\tau)=\psi_{[11]}(\tau)=5q^{8/13}+\cdots$, 
$\psi_{[3]}(\tau)=\psi_{[10]}(\tau)=4q^{5/13}+\cdots$, $\psi_{[4]}(\tau)=\psi_{[9]}(\tau)=4q^{6/13}+\cdots$, 
$\psi_{[5]}(\tau)=\psi_{[8]}(\tau)=10q^{11/13}+\cdots$, $\psi_{[6]}(\tau)=\psi_{[7]}(\tau)=6q^{7/13}+\cdots$.
These $\psi_{[i]}$ form a vector-valued modular form of weight 2 for SL$(2,\bbZ)$ with multiplier $\rho^{(13)}_5$,
so the products $\psi_{[i]}\psi_{[j]}$, $i\le j$, form one of weight 4 whose multiplier is the symmetric square
of $\rho^{(13)}_5$, namely $1\oplus\rho^{(13)}_{12}\oplus\rho^{(13)}_{14}$. Then the
third column of $\Xi(\rho^{(3)}_2\otimes \rho^{(13)}_{14})$ is the vector-valued modular form with
components $2\psi_{1,5}-\psi_{2,3}-\psi_{4,6}$, $\psi_{6,6}-\psi_{2,4}$, $\psi_{0,1}-\psi_{2,6}$, $\psi_{3,5}-\psi_{2,2}$, $\psi_{1,1}-\psi_{4,5}$, $\psi_{0,3}-\psi_{5,6}$, $\psi_{2,5}-\psi_{0,4}$, $\psi_{0,6}-\psi_{1,3}$, 
$\psi_{0,2}-\psi_{1,4}$, $\psi_{1,6}-\psi_{5,5}$, $\psi_{1,2}-\psi_{3,3}$, $\psi_{0,5}-\psi_{3,4}$, $\psi_{4,4}-\psi_{3,6}$, and $\psi_{1,5}+\psi_{2,3}-2\psi_{4,6}$,
where we write $\psi_{i,j}:=\psi_{[i]}\psi_{[j]}\eta^{-8}$.

At $c=8$,
we find $c/24+\mathrm{max}_j\Lambda_{jj} = \frac{3}{13} < 1$, and so we only need to consider 
$$\mathbb{Y}(J(\tau)) = \left(1, Y_2, Y_3, Y_4, 0,0,0,0,0,0,0,0,0,Y_{14},Y_{15},Y_{16},0,0,0,Y_{20},Y_{21},Y_
{22}\right)^\transpose$$
for $Y_2, Y_3, Y_4, Y_{14}, Y_{15}, Y_{16}, Y_{20}, Y_{21}, Y_{22} \in \mathbb N$.

Writing $\bbX(\tau)=\Xi(\tau)\,\bbY(J(\tau)) = q^\lambda \sum_{n=0}^\infty \mathbb{X}_n q^n$, the conditions
$\mathbb{X}_n \in \mathbb{R}^{22}_{\geq 0}$ merely for $n=0, 1$ give the 27 inequalities
{\small
\begin{align*}
0 & \leq 50 Y_1+50 Y_2+50 Y_3-50 Y_4 \displaybreak[1] \\
0 & \leq 7 Y_1-5 Y_2-2 Y_3+11 Y_{14}+2 Y_{15}-4 Y_{16}-35 Y_{20}-20 Y_{21}+22 Y_{22} \displaybreak[1] \\
0 & \leq 4 Y_1-Y_2-3 Y_3+7 Y_{14}-2 Y_{15}-Y_{16}-10 Y_{20}+4 Y_{21}+2 Y_{22} \displaybreak[1] \\
0 & \leq 21 Y_1+11 Y_2-32 Y_3+638 Y_{14}+2 Y_{15}+11 Y_{16}-761 Y_{20}+12 Y_{21}+342 Y_{22} \displaybreak[1] \\
0 & \leq 38 Y_1+10 Y_2-48 Y_3-1386 Y_{14}+3 Y_{15}-6 Y_{16}+514 Y_{20}+80 Y_{21}+733 Y_{22} \displaybreak[1] \\
0 & \leq 27 Y_1-101 Y_2+74 Y_3-5533 Y_{14}+2 Y_{15}+16 Y_{16}-10695 Y_{20}+300 Y_{21}+150 Y_{22} \displaybreak[1] \\
0 & \leq 76 Y_1+20 Y_2+2 Y_3+50 Y_4-1848 Y_{14}+3 Y_{15}+4 Y_{16}-6174 Y_{20}-162 Y_{21}-1555 Y_{22} \displaybreak[1] \\
0 & \leq 20 Y_1+78 Y_3+50 Y_4-682 Y_{14}-4 Y_{15}-14 Y_{16}+4514 Y_{20}+112 Y_{21}+1120 Y_{22} \displaybreak[1] \\
0 & \leq 595 Y_1-425 Y_2-170 Y_3-45177 Y_{14}-6 Y_{15}+56 Y_{16}+184289 Y_{20}-644 Y_{21}-54530 Y_{22} \displaybreak[1] \\
0 & \leq 320 Y_1-236 Y_2-84 Y_3+6468 Y_{14}+9 Y_{15}-48 Y_{16}-34736 Y_{20}-1120 Y_{21}+6998 Y_{22} \displaybreak[1] \\
0 & \leq 112 Y_1+50 Y_2-162 Y_3-649 Y_{14}-2 Y_{15}-15 Y_{16}+21651 Y_{20}+420 Y_{21}-3038 Y_{22} \displaybreak[1] \\
0 & \leq 98 Y_1+34 Y_2-132 Y_3-3010 Y_{14}-4 Y_{15}+6 Y_{16}+17990 Y_{20}-175 Y_{21}+842 Y_{22} \displaybreak[1] \\
0 & \leq 532 Y_1+192 Y_2-724 Y_3+110924 Y_{14}+5 Y_{15}+88 Y_{16}-226308 Y_{20}+481 Y_{21}+45493 Y_{22} \displaybreak[1] \\
0 & \leq 632 Y_1+245 Y_2-877 Y_3-200684 Y_{14}+16 Y_{15}-32 Y_{16}+128651 Y_{20}+2144 Y_{21}+79472 Y_{22} \displaybreak[1] \\
0 & \leq 410 Y_1-1526 Y_2+1116 Y_3-504274 Y_{14}+5 Y_{15}+86 Y_{16}-1555890 Y_{20}+6480 Y_{21}+11658 Y_{22} \displaybreak[1] \\
0 & \leq -2 Y_1-6 Y_2+8 Y_3+12 Y_{14}+3 Y_{15}+4 Y_{16}+42 Y_{20}-2 Y_{21}+13 Y_{22} \displaybreak[1] \\
0 & \leq 19 Y_1-60 Y_2+41 Y_3+1122 Y_{14}+4 Y_{15}-8 Y_{16}+271 Y_{20}+56 Y_{21}-1076 Y_{22} \displaybreak[1] \\
0 & \leq 7 Y_1-18 Y_2+11 Y_3-77 Y_{14}-2 Y_{15}+Y_{16}+189 Y_{20}-28 Y_{21}+194 Y_{22}
\end{align*}
}%
along with some redundant ones.
Some linear programming easily gives upper bounds on all the variables: $Y_2, Y_3 \leq 1, Y_4 \leq 3, Y_{15}, Y_{16} \leq 2$, and $Y_{14},Y_{20},Y_{21},Y_{22} = 0$.
We then easily enumerate all solutions, obtaining 13 possible character vectors. Of these, 9 have components which are identically zero, which is not allowed.
The remaining four have vacuum components as given below.
{\small
\begin{align*}
\bbY_1 & = \left(1,0,0,0,0,0,0,0,0,0,0,0,0,0,0,1,0,0,0,0,0,0\right)^\transpose \displaybreak[1] \\ 
 q^{1/3} \bbX_1(\tau)_{\omega_0} & = 1 + 12 q + 73 q^2 + 346 q^3 + 1390 q^4 + 4956 q^5 + 16715 q^6 + 52982 q^7 +  \cdots
 \displaybreak[1] \\ 
\bbY_2 & = \left(1,0,0,0,0,0,0,0,0,0,0,0,0,0,1,0,0,0,0,0,0,0\right)^\transpose \displaybreak[1] \\ 
 q^{1/3} \bbX_2(\tau)_{\omega_0} & = 1 + 3 q + 22 q^2 + 86 q^3 + 461 q^4 + 1992 q^5 + 8343 q^6 + 30997 q^7 +  \cdots
 \displaybreak[1] \\ 
\bbY_3 & = \left(1,0,0,0,0,0,0,0,0,0,0,0,0,0,1,1,0,0,0,0,0,0\right)^\transpose \displaybreak[1] \\ 
 q^{1/3} \bbX_3(\tau)_{\omega_0} & = 1 + 13 q + 83 q^2 + 372 q^3 + 1460 q^4 + 5112 q^5 + 17053 q^6 + 53651 q^7 +  \cdots
 \displaybreak[1] \\ 
\bbY_4 & = \left(1,0,0,0,0,0,0,0,0,0,0,0,0,0,2,0,0,0,0,0,0,0\right)^\transpose \displaybreak[1] \\ 
 q^{1/3} \bbX_4(\tau)_{\omega_0} & = 1 + 4 q + 32 q^2 + 112 q^3 + 531 q^4 + 2148 q^5 + 8681 q^6 + 31666 q^7 +  \cdots
 \displaybreak[1]
\end{align*}
}%

All four of these possible character vectors have four components equal to 
$$\theta_{[2]}^2\eta^{-8},
\theta_{[0]}\theta_{[1]}\eta^{-8},\theta_{[0]}\theta_{[2]}\eta^{-8},\theta_{[1]}^2\eta^{-8},$$ i.e. identical with components
of the character vector of the lattice VOA for $A_4\oplus A_4$. This is highly suggestive: the extended
Haagerup VOA (at $c=8$) should contain some orbifold of the $A_4\oplus A_4$ lattice VOA.
That subVOA would also have $c=8$, which means the (hypothetical) extended Haagerup VOA would be a finite
extension of that lattice orbifold. Something similar happens for the (still hypothetical) $c=8$
Haagerup VOA, but there the lattice orbifold VOA (which is $\cV_L^+$ for $L=A_{3}{52}[1,\frac{1}{4}]$)
only has $c=4$. So in this sense the extended Haagerup VOA is more accessible than the
Haagerup VOA. Curiously, this the same lattice  $A_{3}52[1,\frac{1}{4}]$ makes an appearance
both in the Haagerup and extended Haagerup.

We now employ Lemma \ref{lem:integrality} to ensure
integrality of the Fourier coefficients. 

\begin{lem}
The vector valued modular forms $\bbY_1, \bbY_2, \bbY_3, \bbY_4$ are integral.
\end{lem}
\begin{proof}
We begin by showing that the columns of $\Xi_n(\rho_{13})$ and $\Xi_n(\rho_5)$ with $\Lambda_{jj} \geq -1/3$ are
themselves integral. (These are the only relevant columns, as all other entries of the $\bbY_i$ are automatically zero.)
To see this, we apply the Lemma to the vector-valued modular form $Q\Xi(\tau) Q^{-1} e_i$ for $i \in 
\{1,2,3,4,14,15,16,20,21,22\}$.
For $i \in \{1,2,3,4,14\}$, the projective kernel is $\pm \Gamma(13)$ with index $1092$, while for $i \in \{15,16\}$
the
projective kernel $\pm \Gamma(5)$ has index $60$, and for $i \in \{20,21,22\}$ the index is 1. Thus it suffices to check
out as far as $8 \cdot 1092 / 24 = 364$.

Next, note that $\bbY_1 - \bbY_i$ is supported on $\rho_{13}$: more precisely, each of the 
differences  $\bbX_1(\tau) - \bbX_i(\tau)$ lies in the $\bbZ$-span of the third and fourth columns
of $\Xi(\rho_2^{(3)}\otimes \rho_{14}^{(13)})$, so this is covered by the previous paragraph.

Finally, we need to see that $\bbX_1(\tau)$ is integral. We observe that the inverse
of $Q$ is almost integral:
\begin{align*} 
Q^{-1} & = \left(\begin{smallmatrix}
\frac{2}{3} & 0 & 0 & 0 & 0 & 0 & 0 & 0 & 0 & 0 & 0 & 0 & 0 & -\frac{1}{3} & \frac{1}{6} & 0 & 0 & 0 & 0 & \frac{1}{6} & \frac{1}{6} &
   \frac{1}{6} \\
 -\frac{1}{3} & 0 & 0 & 0 & 0 & 0 & 0 & 0 & 0 & 0 & 0 & 0 & 0 & \frac{2}{3} & \frac{1}{6} & 0 & 0 & 0 & 0 & \frac{1}{6} & \frac{1}{6} &
   \frac{1}{6} \\
 -\frac{1}{3} & 0 & 0 & 0 & 0 & 0 & 0 & 0 & 0 & 0 & 0 & 0 & 0 & -\frac{1}{3} & \frac{1}{6} & 0 & 0 & 0 & 0 & \frac{1}{6} & \frac{1}{6}
   & \frac{1}{6} \\
 0 & 0 & 0 & 0 & 0 & 0 & 0 & 0 & 0 & 0 & 0 & 0 & 0 & 0 & -\frac{1}{6} & 0 & 0 & 0 & 0 & -\frac{1}{6} & \frac{5}{6} & \frac{1}{6} \\
 0 & 0 & 0 & 0 & 0 & 0 & 0 & 0 & 0 & 0 & 0 & 0 & 0 & 0 & -\frac{1}{6} & 0 & 0 & 0 & 0 & \frac{5}{6} & -\frac{1}{6} & \frac{1}{6} \\
 0 & 0 & 0 & 0 & 0 & 0 & 0 & 0 & 0 & 0 & 0 & 0 & 0 & 0 & -\frac{1}{6} & 0 & 0 & 0 & 0 & -\frac{1}{6} & -\frac{1}{6} & \frac{1}{6} \\
 0 & 0 & 0 & 0 & 0 & 0 & 0 & 0 & 0 & 0 & 0 & 0 & 0 & 0 & 0 & -1 & 0 & 0 & 0 & 0 & 0 & 0 \\
 0 & 0 & 0 & 0 & 0 & 0 & 0 & 0 & 0 & 0 & 0 & 0 & 0 & 0 & 0 & 0 & -1 & 0 & 0 & 0 & 0 & 0 \\
 0 & 0 & 0 & 0 & 0 & 0 & 0 & 0 & 0 & 0 & 0 & 0 & 0 & 0 & 0 & 0 & 0 & -1 & 0 & 0 & 0 & 0 \\
 0 & 0 & 0 & 0 & 0 & 0 & 0 & 0 & 0 & 0 & 0 & 0 & 0 & 0 & 0 & 0 & 0 & 0 & -1 & 0 & 0 & 0 \\
 0 & 0 & 0 & 0 & 0 & 0 & 0 & 0 & 0 & -1 & 0 & 0 & 0 & 0 & 0 & 0 & 0 & 0 & 0 & 0 & 0 & 0 \\
 0 & 0 & 0 & 0 & 0 & 0 & -1 & 0 & 0 & 0 & 0 & 0 & 0 & 0 & 0 & 0 & 0 & 0 & 0 & 0 & 0 & 0 \\
 0 & 0 & 0 & 0 & 1 & 0 & 0 & 0 & 0 & 0 & 0 & 0 & 0 & 0 & 0 & 0 & 0 & 0 & 0 & 0 & 0 & 0 \\
 0 & 0 & 0 & 0 & 0 & 0 & 0 & 1 & 0 & 0 & 0 & 0 & 0 & 0 & 0 & 0 & 0 & 0 & 0 & 0 & 0 & 0 \\
 0 & 0 & 0 & 1 & 0 & 0 & 0 & 0 & 0 & 0 & 0 & 0 & 0 & 0 & 0 & 0 & 0 & 0 & 0 & 0 & 0 & 0 \\
 0 & 0 & -1 & 0 & 0 & 0 & 0 & 0 & 0 & 0 & 0 & 0 & 0 & 0 & 0 & 0 & 0 & 0 & 0 & 0 & 0 & 0 \\
 0 & 0 & 0 & 0 & 0 & 0 & 0 & 0 & 0 & 0 & 1 & 0 & 0 & 0 & 0 & 0 & 0 & 0 & 0 & 0 & 0 & 0 \\
 0 & 0 & 0 & 0 & 0 & 0 & 0 & 0 & 0 & 0 & 0 & -1 & 0 & 0 & 0 & 0 & 0 & 0 & 0 & 0 & 0 & 0 \\
 0 & -1 & 0 & 0 & 0 & 0 & 0 & 0 & 0 & 0 & 0 & 0 & 0 & 0 & 0 & 0 & 0 & 0 & 0 & 0 & 0 & 0 \\
 0 & 0 & 0 & 0 & 0 & -1 & 0 & 0 & 0 & 0 & 0 & 0 & 0 & 0 & 0 & 0 & 0 & 0 & 0 & 0 & 0 & 0 \\
 0 & 0 & 0 & 0 & 0 & 0 & 0 & 0 & 0 & 0 & 0 & 0 & -1 & 0 & 0 & 0 & 0 & 0 & 0 & 0 & 0 & 0 \\
 0 & 0 & 0 & 0 & 0 & 0 & 0 & 0 & -1 & 0 & 0 & 0 & 0 & 0 & 0 & 0 & 0 & 0 & 0 & 0 & 0 & 0
 \end{smallmatrix}\right).
\end{align*}
We see from the locations of denominators in $Q^{-1}$ (and our earlier observation about the integrality of the
matrices $\Xi_n(\rho)$) that it is only the first six entries of $\bbX_1(\tau)$ which might not be integral.
Consider first $\bbX_1(\tau)_{\alpha_1}$, the 4th component (which also equals the fifth and sixth
components). Note $Q \bbY_1 = e_1 - e_3 + e_{15} + e_{22}$
, and compute 
\begin{align*}
(Q^{-1} \Xi_n Q \bbY_1)_4 
	= \frac{1}{6} \big(
       & -(\Xi_n)_{15,1}+(\Xi_n)_{15,3}-(\Xi_n)_{15,15}-(\Xi_n)_{15,22} \\
       & -(\Xi_n)_{20,1}+(\Xi_n)_{20,3}-(\Xi_n)_{20,15}-(\Xi_n)_{20,22} \\
       & -(\Xi_n)_{21,1}+(\Xi_n)_{21,3}-(\Xi_n)_{21,15}-(\Xi_n)_{21,22} \\
       & +(\Xi_n)_{22,1}-(\Xi_n)_{22,3}+(\Xi_n)_{22,15}+(\Xi_n)_{22,22} \big)\\
\intertext{and observe that most of these vanish as $\Xi_n$ is block diagonal, obtaining}
(\bbX_1)_{\alpha_1}=(Q^{-1} \Xi_n Q \bbY_1)_4 = \frac{1}{6}\big(& -(\Xi_n)_{15,15}+(\Xi_n)_{22,22}\big).
\end{align*}
These are the coefficients of a (scalar) modular function for $\Gamma=\pm\Gamma(5)$, so we can apply Lemma \ref
{lem:integrality} with $m=60$, $k=8$. After checking explicitly that the first 20 values of $(Q^{-1} \Xi_n Q
\bbY_1)_4$ are integral, this ensures that $\bbX_1(\tau)_{\alpha_1}$ is integral. 


Now consider $\bbX_1(\tau)_{\omega_0}$, the 1st component. 
We need to show that $(Q^{-1} \Xi_n Q \bbY_1)_1$ is integral. To do this, we take advantage of the fact that
$(Q^{-1})_1 + (Q^{-1})_4 \pmod 1 = \frac{2}{3} e_1 + \frac{2}{3} e_{14} + \frac{1}{3} e_{22}$,
and that we have already shown  $\Xi_n Q \bbY_1$ and $(Q^{-1} \Xi_n Q \bbY_1)_4$ are integral.
We then see
\begin{align*}
\left(\frac{2}{3} e_1 + \frac{2}{3} e_{14} + \frac{1}{3} e_{22}\right) \Xi_n Q \bbY_1 
	& = \frac{2}{3} \left((\Xi_n)_{1,1} - (\Xi_n)_{1,3} + (\Xi_n)_{14,1} - (\Xi_n)_{14,3}\right)
	  + \frac{1}{3} (\Xi_n)_{22,22}
\end{align*}
is a modular function for $\Gamma = \pm \Gamma(13)$. Again, Lemma \ref
{lem:integrality} with $m=1092$, $k=8$ allows us to check the first 364 coefficients to ensure that $\bbX_1(\tau)_
{\omega_0}$ is integral.

Finally $(Q^{-1})_1 - (Q^{-1})_2$ and $(Q^{-1})_1 - (Q^{-1})_3$ are integral and supported in entries $1, 14$, and the
corresponding columns of $\Xi(\tau)$ are integral, so $\bbX_1(\tau)_{\omega_i}$ are all integral.
\end{proof}

Multiplying any character vector at $c=8$ by $J(\tau)^{1/3}$ resp.\ $J(\tau)^{2/3}$
will give a character vector at $c=16$ resp.\ $c=24$. But there should be many more as
$c$ grows, and knowing other candidates could be important if all 4 candidates at $c=8$ fail
to be realised by a vertex operator algebra.
At $c = 16$ we find $c/24+\mathrm{max}_j\Lambda_{jj} = \frac{4}{5} < 1$, so we consider
$$\mathbb{Y}(J(\tau)) = \left(1, Y_2, Y_3, Y_4, Y_5, Y_6, Y_7, Y_8, Y_9, 0,0,0,0, Y_{14}, Y_{15}, Y_{16}, Y_{17} , 0, Y_{19}, Y_{20}, Y_{21}, Y_{22}\right)^\transpose.$$
Again the conditions $\mathbb{X}_n \in \mathbb{R}^{22}_{\geq 0}$ for $n=0,1$ suffice to obtain finitely many cases; we
obtain inequalities
{\small
\begin{align*}
0 & \leq 13 Y_1-13 Y_2-Y_{14}-5 Y_{15}+8 Y_{16}-21 Y_{17}+2 Y_{19}+3 Y_{20}-9 Y_{21}-Y_{22} \displaybreak[1] \\
0 & \leq 2119 Y_1-1547 Y_2-572 Y_3+2 Y_{14}+351 Y_{15}-896 Y_{16}-728 Y_{17}+832 Y_{19}-28 Y_{20}-5174 Y_{21}+Y_{22} \displaybreak[1] \\
0 & \leq 975 Y_1-715 Y_2-260 Y_3+6 Y_{14}-254 Y_{15}-16 Y_{16}-2058 Y_{17}-235 Y_{19}-60 Y_{20}+750 Y_{21}+Y_{22} \displaybreak[1] \\
0 & \leq 39 Y_1+13 Y_2-52 Y_3-3 Y_{14}+26 Y_{15}-16 Y_{16}-78 Y_{17}-26 Y_{19}+2 Y_{20}+65 Y_{21}+Y_{22} \displaybreak[1] \\
0 & \leq 339 Y_1+103 Y_2+25 Y_3+175 Y_4+175 Y_5+175 Y_6+98 Y_7+20 Y_8+880 Y_9-4 Y_{14}  \\
	& \qquad +28 Y_{15}+32 Y_{16}+266 Y_{17}-212 Y_{19}-30 Y_{20}-452 Y_{21}-2 Y_{22} \displaybreak[1] \\
0 & \leq 103 Y_1+27 Y_2+337 Y_3+175 Y_4+175 Y_5+175 Y_6+98 Y_7+20 Y_8+880 Y_9-2 Y_{14} \\
	& \qquad -86 Y_{15}-144 Y_{16}-882 Y_{17}+146 Y_{19}+20 Y_{20}+326 Y_{21}+2 Y_{22} \displaybreak[1] \\
0 & \leq 175 Y_1+175 Y_2+175 Y_3+817 Y_4-175 Y_5-175 Y_6-98 Y_7-20 Y_8-880 Y_9 \displaybreak[1] \\
0 & \leq 175 Y_1+175 Y_2+175 Y_3-175 Y_4+817 Y_5-175 Y_6-98 Y_7-20 Y_8-880 Y_9 \displaybreak[1] \\
0 & \leq 175 Y_1+175 Y_2+175 Y_3-175 Y_4-175 Y_5+817 Y_6-98 Y_7-20 Y_8-880 Y_9 \displaybreak[1] \\
0 & \leq 375 Y_1+375 Y_2+375 Y_3-375 Y_4-375 Y_5-375 Y_6+56 Y_7-50 Y_8+3300 Y_9 \displaybreak[1] \\
0 & \leq 2025 Y_1+2025 Y_2+2025 Y_3-2025 Y_4-2025 Y_5-2025 Y_6-1372 Y_7+164 Y_8+8624 Y_9 \displaybreak[1] \\
0 & \leq 227375 Y_1+227375 Y_2+227375 Y_3-227375 Y_4-227375 Y_5-227375 Y_6+46648 Y_7+850 Y_8-44107350 Y_9 \displaybreak[1] \\
0 & \leq 6084 Y_1-4394 Y_2-1690 Y_3-6 Y_{14}-374 Y_{15}+1702 Y_{16}-49168 Y_{17}+1564 Y_{19}+88 Y_{20}-4582 Y_{21}-2 Y_{22} \displaybreak[1] \\
0 & \leq 115765 Y_1-83993 Y_2-31772 Y_3+8 Y_{14}+7956 Y_{15}-28112 Y_{16}-79196 Y_{17} \\
	& \qquad +34684 Y_{19}-308 Y_{20}-367952 Y_{21}+Y_{22} \displaybreak[1] \\
0 & \leq 67704 Y_1-49114 Y_2-18590 Y_3+22 Y_{14}-6878 Y_{15}-906 Y_{16}-309708 Y_{17}\\
	& \qquad -12288 Y_{19}-780 Y_{20}+70830 Y_{21}+2 Y_{22} \displaybreak[1] \\
0 & \leq 3042 Y_1-2197 Y_2-845 Y_3+10 Y_{14}-26 Y_{15}-460 Y_{16}+16380 Y_{17}-2392 Y_{19}+44 Y_{20}+3770 Y_{21}-Y_{22} \displaybreak[1] \\
0 & \leq 455 Y_1+182 Y_2-637 Y_3-68 Y_{15}-160 Y_{16}+980 Y_{17}+92 Y_{19}+28 Y_{20}+1160 Y_{21}-Y_{22} \displaybreak[1] \\
0 & \leq 312 Y_1+104 Y_2-416 Y_3-3 Y_{14}-65 Y_{15}+64 Y_{16}+2002 Y_{17}+234 Y_{19}+30 Y_{20}-390 Y_{21} \displaybreak[1] \\
0 & \leq 13 Y_1+13 Y_2-26 Y_3+2 Y_{14}+6 Y_{15}+19 Y_{16}-28 Y_{17}+24 Y_{19}-4 Y_{20}+10 Y_{21}+Y_{22} \displaybreak[1] \\
0 & \leq 8450 Y_1+3211 Y_2-11661 Y_3-20 Y_{14}+1768 Y_{15}-1320 Y_{16}-52780 Y_{17}\\
	& \qquad -3484 Y_{19}+44 Y_{20}+19604 Y_{21}+2 Y_{22} \displaybreak[1] \\
0 & \leq 78 Y_1-260 Y_2+182 Y_3-6 Y_{14}+14 Y_{15}+98 Y_{16}+924 Y_{17}+80 Y_{19}-20 Y_{20}+546 Y_{21} \displaybreak[1] \\
0 & \leq 403 Y_1-1443 Y_2+1040 Y_3+4 Y_{14}+208 Y_{15}+523 Y_{16}-3796 Y_{17}-884 Y_{19}+76 Y_{20}+156 Y_{21}+Y_{22} \displaybreak[1] \\
0 & \leq 26 Y_1-91 Y_2+65 Y_3+2 Y_{14}+34 Y_{15}-28 Y_{16}+56 Y_{17}+92 Y_{19}+58 Y_{21}-Y_{22} \displaybreak[1] \\
0 & \leq 1170 Y_1-4342 Y_2+3172 Y_3-5 Y_{14}-351 Y_{15}+280 Y_{16}+20475 Y_{17}+169 Y_{19}+57 Y_{20}-8385 Y_{21}+2 Y_{22}
\end{align*}
}%
with 179,459 solutions. All appear to have positive integral Fourier coefficients for many (and probably all) terms.
This time, Lemma 10.2 would require checking about twice as many coefficients for integrality as was necessary
for $c=8$. Although this is probably possible, enough effort is involved that we have not done this.

At $c  =24$  we find $c/24+\mathrm{max}_j\Lambda_{jj} = 1$, so we consider
\begin{align*}
\mathbb{Y}(J(\tau)) & = \left(J(\tau) + Y_1, Y_2, Y_3, Y_4, Y_5, Y_6, Y_7, Y_8, Y_9, \right. \\
& \qquad \qquad \left. Y_{10}, Y_{11}, Y_{12}, Y_{13}, Y_{14}, Y_{15}, Y_{16}, Y_{17} , Y_{18}, Y_{19}, Y_{20}, Y_{21},
Y_{22}\right)^\transpose,
\end{align*}
but this time just with $Y_i \in \mathbb{C}$. The requirement that the first component $\Xi(\tau)\,\bbY(J(\tau))$ has the
strictly lowest leading exponent forces $Y_{10} = - \frac{1}{6}, Y_{18} = - \frac{2}{3}$, and $Y_{21} = -1$.
Now, sadly, the conditions $\mathbb{X}_n \in \mathbb{R}^{22}_{\geq 0}$ do not appear to cut out a bounded region, no matter how high an $n$ we consider. (In particular, $Y_1, \ldots, Y_6$ are unbounded.) However the conditions $\mathbb{X}'_{n;j} \geq 0$ for $n=0,1$ do cut out a bounded region.
We cannot enumerate the points however (the naive upper bound we have on its volume in the $Y_i$ coordinate system is around $10^{43}$), and  the collection of solutions may shrink further as we consider $\mathbb{X}'_{n;j} \geq 0$ for larger $n$.
Nevertheless, it is possible to find new individual solutions, for example
\begin{align*}
\mathbb{Y}(J(\tau)) & = \left(J(\tau) + \frac{519}{2}, - \frac{23}{6}, \frac{83}{6}, \frac{625}{6}, \frac{625}{6}, \frac{625}{6}, 
\frac{1}{2}, \frac{19}{3},  \right. \\
& \qquad \qquad \left. 12, - \frac{1}{6}, 1, 16, - \frac{10}{3}, 
\frac{77}{3}, -4, 5, \frac{302}{3}, - \frac{2}{3}, \frac{1}{3}, 10, -1, 49\right)^\transpose
\end{align*}
which gives non-negative integral $\mathbb{X}'_{n;j}$ at least up to $n=50$. 

\appendix
\section{Some consequences of Lemma \ref{lem:bounding-orbit-size-from-N}}

We record here some additional consequences of Lemma \ref{lem:bounding-orbit-size-from-N}, which although unneeded for
the present argument, may prove useful to others.
\begin{cor}
\label{lem:bounding-N-from-orbit-size}
If the full Galois orbit of some $x\in\Phi$ has cardinality $k \leq 6$, then the root of unity $T_{xx}$
has order dividing some number in the set $\cN_k$, where
\begin{align*}
\cN_1 & =\{2^3 \cdot 3\} \\
\cN_2 & =\{2^3 \cdot 3 \cdot 5, 2^4 \cdot 3\} \\
\cN_3 & =\{2^3 \cdot 3^2, 2^3 \cdot 3 \cdot 7\} \\
\cN_4 & =\{2^5 \cdot 3, 2^4 \cdot 3 \cdot 5\} \\
\cN_5 & =\{2^3 \cdot 3 \cdot 11\} \\
\cN_6 & =\{2^4 \cdot 3^2, 2^4 \cdot 3 \cdot 7, 2^3 \cdot 3^2\cdot5,2^3 \cdot 3 \cdot 5 \cdot 7, 2^3 \cdot 3 \cdot 13\}.
\end{align*}
\end{cor}
\begin{proof}
Clearly the formula for $k(N_x)$ in Lemma \ref{lem:bounding-orbit-size-from-N} is increasing with respect to the
factorization of $N_x$. Moreover $N_x$ can not be divisible by any prime $p$ larger than $13$, as otherwise $k(N_x) \ge
(p-1)/2 > 6$.
Thus we just need to check small exponents in $N_x = 2^{\mu_2} 3^{\mu_3} 5^{\mu_5} 7^{\mu_7} 11^{\mu_{11}} 13^{\mu_{13}}$.
\end{proof}

The Mathematica notebook {\tt ConductorsForOrbitsSize.nb} available with the {\tt arXiv} sources of this article readily computes $\cN_k$ for values of
$k$ up to
several hundred.

\begin{cor}
Let $k_x$ be the size of the full Galois orbit of an object $x$ and $N_x$ be the order of $T_{xx}$. Then for any $\delta > 0$ we have
$$N_x \leq C_\delta k_x^{1+\delta}$$
where
$$C_\delta = 24 \prod_{p\in P_\delta} p \left(\frac{2}{p-1}\right)^{1+\delta}$$
where the product is taken over the finite set $$P_\delta= \left\{ 3 < p : \text{$p$ is prime and $p
\left(\frac{2}{p-1}\right)^{1+\delta} > 1$} \right\}.$$
\end{cor}
(The set $P_\delta$ is certainly finite as all such primes are less than $\max\{7, 1+ 2 \left(\frac{11}{5}\right)^{1/\delta}\}$.)

\begin{proof}
Write $N_x=\prod_pp^{\mu_p}$ as before. We have
$$\frac{N_x}{k_x^{1+\delta}} = \prod_{p | N_x} R_p$$
where
\begin{align*}
R_2 & = \begin{cases}
2 & \text{if $\mu_2 = 1$} \\
4 & \text{if $\mu_2 = 2$} \\
2^{3+\delta(3-\mu_2)} & \text{if $\mu_2 \geq 3$}
\end{cases} \\
\intertext{and}
R_p & = p^{1+\delta(1-\mu_p)} \left(\frac{2}{p-1}\right)^{1+\delta}.
\end{align*}
Thus in the worst case $\mu_2 = 3$ and $\mu_p = 1$ for each other $p | N_x$. 
When $\mu_p = 1$, $R_p$ simplifies to $p \left(\frac{2}{p-1}\right)^{1+\delta}$.
We then have
\begin{align*}
\frac{N_x}{k_x^{1+\delta}} & = R_2 R_3 \prod_{3<p | N_x} R_p \\
 & \leq 24 \prod_{p \in P_\delta} p \left(\frac{2}{p-1}\right)^{1+\delta} \\
 & = C_\delta. \qedhere
\end{align*}
\end{proof}

The rank of a modular tensor category is the sum of the sizes of the Galois orbits of objects, while the exponent is the least common multiple of the orders of the eigenvalues of $T$, so while we have close-to-linear bounds on the conductor on each orbit, it is still possible to have exponential growth of $\operatorname{ord}(T)$ relative to the rank, as for $\operatorname{Rep} D S_n$.

Incidentally, for all odd primes the smallest irrep of SL$(2,\bbZ_{p^\nu})$ with conductor $p^\nu$
for $\nu\ge 2$
has dimension $(p^2-1)p^{\nu-2}$. The smallest irrep with conductor $2^\nu$ for $\nu\ge 4$ has
dimension $3\cdot 2^{\nu-4}$. 

\section{Explicit matrices for some irreps of \texorpdfstring{$SL(2,\bbZ)$}{SL(2,Z/13)}}
\label{appendix:matrices}

The representations we are interested in both lie in the \textit{principal series} of SL($2,\bbZ_p$).
In particular, write $B$ for the (Borel) subgroup of upper-triangular matrices
$\begin{psmallmatrix} a & b \\ 0 & a^{-1} \end{psmallmatrix}$. Each irrep $\lambda$ of $\bbZ^\times_p\cong
\bbZ_{p-1}$ extends to $B$ by $\lambda\begin{psmallmatrix} a & b \\ 0 & a^{-1} \end{psmallmatrix}=\lambda
(a)$. Denote by $\rho^{(p);\lambda}$ the induced representation Ind$_B^{SL(2,\bbZ_p)}
\lambda$ --- it will be $p+1$-dimensional. Then $\rho^{(p);\lambda}\cong \rho^{(p);\bar{\lambda}}$ is  irreducible iff $\lambda^2\ne 1$. By contrast, 
$\rho^{(p);1}$ is the direct sum of 1 and an irrep called the Steinberg representation, while $\rho^{(p);\lambda}$ for the order-2 $\lambda$ is the
direct sum of two $(p+1)/2$-dimensional irreps. Coset representatives 
for SL$(2,p)/B$ are $\begin{psmallmatrix} 1 & 0 \\ j & 1\end{psmallmatrix}$ and $
\begin{psmallmatrix} 0 & -1 \\ 1 & 1\end{psmallmatrix}$, and using this it is easy to work out not merely the characters of
$\rho^{(p);\lambda}$, but explicit matrices as well. 

The modular data of the centre of the extended Haagerup has two building blocks: the conductor-5 
irrep $\rho_8^{(5)}$ and the conductor-13 irrep $\rho_{14}^{(13)}$.
The irrep $\rho^{(5)}_8$ is the Steinberg representation for SL$(2,\bbZ_5)$,
while the other irrep, $\rho^{(13)}_{14}$, is $\rho^{(13);\lambda}$ for the 
 (unique up to complex conjugate) order-3 $\lambda$. It is thus easy to work out
explicit matrix realisations.
First, $\rho_8^{(5)}$ is generated by matrices $T_8^{(5)}=\mathrm{diag}(1,\xi_5,\xi_5^2,\xi_5^3,\xi_5^4)$ and
$$S^{(5)}_8 := \frac{1}{5}\left(\begin{matrix}-1&            -6&  -6&  -6&   -6\\ 
 -1& -2c-c'&                      2c'&2c& -c-2c'\\ -1 &   2c'& -c-2c'& -2c-c'&      2c\\
 -1&         2c& -2c-c'& -c-2c'&     2c'\\ -1& -c-2c'&        2c&          2c'& -2c-c'\end{matrix}\right)$$
where we write $c=2\cos(2\pi/5)$ and $c'=2\cos(4\pi/5)$.

Likewise, $\rho^{(13)}_{14}$ is generated by $14\times 14$ matrices $S^{(13)}_{14}$ and $T^{(13)}_{14}$.
Label their rows/columns by $0,1,2,\ldots,12,0'$ in that order. Then $\left(T^{(13)}_{14}\right)_{00}=\left(T^{(13)}_{14}\right)_{0'0'}=1$ and $\left(T^{(13)}_{14}\right)_{ll}=\xi_{13}^l$ for each $1\le l\le 12$. Write $c_j=2\cos(2\pi j/13)$. Define
vectors $\varepsilon=(1,1,-1,-1,1,-1)$ and $\varepsilon'=(1,1,1,-1,1,1)$ and quantities
$s(l)=(c_l-c_{2l}-c_{3l}+c_{5l})/13$, $s'(l)=(1+3c_l+3c_{5l})/13$, $t(l)=(2-c_l+c_{2l}-c_{4l})/13$, and $t'(l)=(2c_{2l}-c_{3l}-c_{5l})/13$. Then
$$S_{00}=s(1)\,,\ S_{00'}=s(2)\,,\ S_{0'0}=-s(4)\,,\ S_{0'0'}=-s(1)\,,$$
$$S_{l^2,0}=\varepsilon_l\,s(l)\,,\ S_{0,l^2}=\varepsilon_l\,s'(l)\,,\ S_{2l^2,0}=\varepsilon'_l\,s(4l)\,,\ S_{0,2l^2}=\varepsilon'_l\,s'(4l)\,,$$
$$S_{l^2,0'}=\varepsilon_l\,s(2l)\,,\ S_{0',l^2}=-\varepsilon_l\,s'(4l)\,,\ S_{2l^2,0'}=\varepsilon'_l\,s(5l)\,,\ S_{0',2l^2}=-\varepsilon'_l\,s'(3l)\,,$$
$$S_{l^2,m^2}=\varepsilon_l\,\varepsilon_m\,t(lm)\,,\ S_{l^2,2m^2}=S_{2m^2,l^2}=\varepsilon_l\,\varepsilon'_m\,t'(lm)\,,\ S_{2l^2,2m^2}=\varepsilon'_l\,\varepsilon'_m\,t(2lm)\,,$$
where subscripts in  $S_{l^2,m^2}$ etc are taken mod 13, and $l,m$ run over all numbers $1,2,\ldots,6$.
The parameters $l,m$ parametrise the quadratic residues and nonresidues mod 13, which behave
slightly differently.

Curiously, the doubles of the even parts of both the Haagerup and Asaeda--Haagerup subfactors
are likewise built from the principal series, for $p=13$ and $p=17$ respectively, specifically from one of the
$(p+1)/2$-dimensional irreps in $\rho^{(p);\lambda}$ for the order-2 $\lambda$.  

It would be interesting to investigate the possibility of fitting the modular data
of the extended Haagerup into an infinite sequence.  This would be somewhat analogous to doing it for the
Haagerup. The latter was done in \cite{MR2837122}, but what made that possible was that
there was already an infinite family to which the Haagerup hypothetically belonged \cite{MR1832764},
and the first several subfactors in that sequence were already known to exist \cite{MR2837122}.
Doing this for the extended Haagerup would be a much greater challenge, but a very interesting one!

In particular, we learnt above that
the SL$(2,\bbZ)$-representation for the extended Haagerup is isomorphic to $\rho_{13}\oplus \rho_5\oplus 1\oplus 1
\oplus 1$, where $\rho_5$ is the Steinberg representation of SL$(2,\bbZ_5)$
and  $\rho_{13}$ lies in the principal series of SL$(2,\bbZ_{13})$. So we may look
for modular data isomorphic to $\rho_p\oplus \rho_r\oplus n\,1$ for some $n\in
\bbZ_{\ge 0}$ and some primes $r,p$, where again $\rho_r$ is Steinberg and $\rho_p$ lies in the principal series (and perhaps corresponds to the unique $\lambda\in\bbZ^\times_p$ of
maximal odd order). In the expression for $S$
in Theorem \ref{thm:S}, there are six awkward submatrices, namely $U,V,W,A,B,C$.
But thanks to Lemma \ref{lem:basic}, $W$ resp. $A,B,C$ can be read off from 
$\rho_r \left(\begin{smallmatrix}0 & -1 \\ 1 & 0\end{smallmatrix}\right)$ resp.\ $\rho_p \left(\begin{smallmatrix}0 & -1 \\ 1 & 0\end{smallmatrix}\right)$, 
and as we explained in this appendix those matrices are readily computed. Moreover
unitarity of $S$ forces $n=r-2$. So
only the symmetric $n$-by-$n$ matrix $U$ and the $n$-by-$(p{-}1)$ matrix $V$
need to be identified. They are directly obtained from $\rho_r \left(\begin{smallmatrix}0 & -1 \\ 1 & 0\end{smallmatrix}\right)$ and $\rho_p \left(\begin{smallmatrix}0 & -1 \\ 1 & 0\end{smallmatrix}\right)$ by the change-of-basis
matrix we call $Q$ and, as we see in Section \ref{sec:endgame}, $Q$ takes a very simple form
for the extended Haagerup.
We haven't pursued this  any further.

\section*{Acknowledgements}
Scott Morrison was supported by a Discovery Early Career Research Award DE120100232, Discovery Projects `Subfactors and
symmetries' DP140100732 and `Low dimensional categories' DP160103479 from the Australian Research Council, and DOD-DARPA grant HR0011-12-1-0009.
Terry Gannon was supported by an NSERC Discovery grant; he thanks ANU for its generosity during two very pleasant visits. We thank David Evans, Pinhas Grossman and Kevin Walker for interesting discussions. Parts of this paper were written during visits to the Banff International Research Station
and the Hausdorff Institute for Mathematics, and we thank both for providing such stimulating environments.


\bibliographystyle{alpha}
\bibliography{bibliography/bibliography}

\end{document}